\documentclass[reqno]{amsart}
\usepackage[left=2cm,right=2cm,top=3cm,bottom=3cm]{geometry}
\usepackage{amssymb}
\usepackage{amsmath}
\usepackage{graphicx}
\usepackage{euscript}
\usepackage{cite}
\usepackage{leftindex}
\usepackage{color}
\usepackage{tikz-cd}
\usetikzlibrary{tqft}
\usepackage{mathtools}
\usepackage{enumerate}
\usepackage{hyperref}
\hypersetup{
	colorlinks=true,       			% false: boxed links; true: colored links
	linkcolor=blue,          			% color of internal links
	citecolor=blue,		 		% color of links to bibliography
	filecolor=blue,      				% color of file links
	urlcolor=blue           			% color of external links
}
\usepackage{amsthm}
\usepackage{tikz}
\usepackage{amscd}
\usepackage{latexsym}
\usepackage{epsfig}
\usepackage{caption}
\usepackage{rotating}
\usepackage{bigints} 
\usepackage{longtable}
\newcommand{\Z}{\mathbb{Z}}

\newcommand{\R}{\mathbb{R}}
\newcommand{\C}{\mathbb{C}}

\newcommand{\T}{\mathbb{T}}
\newcommand{\g}{\mathfrak{g}}
\renewcommand{\t}{\mathfrak{t}}

\newcommand\longmapsfrom{\mathrel{\reflectbox{\ensuremath{\longmapsto}}}}

\usepackage{tcolorbox}
\usepackage{lastpage}
\newcommand{\sll}[1]{\mkern-4mu\mathbin{/\mkern-5mu/}_{\mkern-4mu{#1}}}

\newcommand{\h}{\mathfrak{h}}

\newcommand{\z}{\mathfrak{z}}

\DeclareMathOperator{\ad}{ad}
\DeclareMathOperator{\Ad}{Ad}
\newcommand{\pr}{\mathrm{pr}}
\DeclareMathOperator{\rank}{rank}

\DeclareMathOperator{\Stab}{Stab}

\newcommand\junk[1]{}

\numberwithin{equation}{section}

\newtheorem{theorem}{Theorem}[section]
\newtheorem*{theorem*}{Theorem}
\newtheorem{proposition}[theorem]{Proposition}
\newtheorem{corollary}[theorem]{Corollary}
\newtheorem{lemma}[theorem]{Lemma}

\theoremstyle{definition}
\newtheorem{definition}[theorem]{Definition}
\newtheorem{example}[theorem]{Example}
\newtheorem{remark}[theorem]{Remark}

\title[Generalized Gelfand--Cetlin data and abelianization]{Generalized Gelfand--Cetlin data and abelianization}

\author[P. Crooks]{Peter Crooks}
\address{Department of Mathematics and Statistics \\ Utah State University \\ Logan, Utah, USA}
\email{peter.crooks@usu.edu}

\author[M. Pound]{Mitchell Pound}
\address{Department of Mathematics and Statistics \\ Utah State University \\ Logan, Utah, USA}
\email{mitchell.pound@usu.edu}

\subjclass{53D20 (primary); 22E46 (secondary)}
\keywords{abelianization, Hamiltonian action, symplectic reduction}

\begin{document}
	
	\begin{abstract}
		The Hamiltonian-geometric analogue of restricting a complex representation of a compact connected Lie group $G$ to a maximal torus $T\subset G$ is symplectic implosion, as developed by Guillemin--Jeffrey--Sjamaar. The first author and Weitsman studied the effect on this analogy of replacing $T$ with a Gelfand--Cetlin torus. In particular, they identified symplectic quotients by $G$ at generic levels in $\g^*_{\text{reg}}$ with symplectic quotients by a Gelfand--Cetlin torus. On the other hand, the applicability of this result is limited to Hamiltonian $G$-spaces whose moment images intersect $\g^*_{\text{reg}}$. The purpose of the present work is to remove that constraint altogether.
		
		We construct generalized Gelfand--Cetlin data for each orbit-type stratum $\Sigma\subset\g^*$. Each datum determines a Gelfand--Cetlin torus and a Hamiltonian action of that torus on an open subset of every connected Hamiltonian $G$-space whose principal orbit-type stratum is $\Sigma$. Our generalized abelianization theorem gives canonical isomorphisms of stratified symplectic spaces between reductions by $G$ at generic levels in $\Sigma$ and corresponding reductions by the Gelfand--Cetlin torus. Using symplectic implosion, we extend the Gelfand--Cetlin abelianizations of the first author and Weitsman to non-regular elements of $\g^*$. We subsequently illustrate this theorem by using the first author and Mayrand's theory of symplectic reduction along a submanifold to construct a canonical family of Hamiltonian $G$-spaces whose moment images lie entirely outside the regular locus. The positive-dimensional reductions of these spaces are realized as reductions by Gelfand--Cetlin tori.
	\end{abstract}
	\maketitle
	
	%%%%%%%%%%%%%%%%%%%%%%%%%%%%%%%%%%%%%%%%%%%%%%%%%%%%%%%%%%%%%%%%%%%
	%%%%%%%%%%%%%%%%%%%%%%%  SECTION 1  %%%%%%%%%%%%%%%%%%%%%%%%%%%%%%%
	%%%%%%%%%%%%%%%%%%%%%%%%%%%%%%%%%%%%%%%%%%%%%%%%%%%%%%%%%%%%%%%%%%%
	\section{Introduction}
	\subsection{Broad context}
	In 1974, Marsden and Weinstein developed an approach to symmetry reduction in symplectic geometry \cite{marsden1974Reduction}. Their reduction procedure is often studied in the context of a compact connected Lie group $G$ with Lie algebra $\g$.
	A second ingredient is a Hamiltonian $G$-space $(M,\omega,\mu)$ with moment map $\mu:M \longrightarrow \g^*$. 
	One fixes an element $\xi \in \g^*$, with $G$-stabilizer denoted $G_\xi \subset G$.
	Assume that $G_\xi$ acts freely on the level set $\mu^{-1}(\xi)$, in which case $\mu^{-1}(\xi)$ is a submanifold of $M$.
	The orbit space $\mu^{-1}(\xi)\slash G_\xi$
	is a smooth manifold carrying a unique symplectic form
	$\omega_\xi \in \Omega^2(\mu^{-1}(\xi) \slash G_\xi)$ satisfying $\pi^*\omega_\xi = j^*\omega$, where \mbox{$\pi:\mu^{-1}(\xi) \longrightarrow \mu^{-1}(\xi)\slash G_\xi$} is the quotient and $j:\mu^{-1}(\xi)\longrightarrow M$ is the inclusion.
	One calls 
	$$
	M\sll{\xi} G \coloneqq \left(\mu^{-1}(\xi) \slash G_\xi , \, \omega_\xi \right)
	$$
	the \textit{symplectic} or \textit{Marsden--Weinstein reduction} \cite{marsden1974Reduction} of $M$ by $G$ at level $\xi \in \g^*$.
	
	Marsden--Weinstein reduction has been central to the last 50 years of progress in Hamiltonian symplectic geometry. Generalizations of this process include Marsden--Ratiu reduction \cite{mar-rat:86}, Mikami--Weinstein reduction \cite{mik-wei:88}, the stratified symplectic quotients of Sjamaar--Lerman \cite{sjamaar1991Stratified}, and symplectic reduction along a pre-Poisson submanifold \cite{crooks2022Symplectic}. On the other hand, variants of Marsden--Weinstein reduction include symplectic cutting \cite{ler:95} and symplectic implosion \cite{gui-jef-sja:02}. The latter is a systematic approach to ``abelianizing'' the symplectic quotients of a Hamiltonian $G$-space.
	
	\subsection{Abelianization in Hamiltonian symplectic geometry}\label{sec:intro-abelianization}
	In the interest of clarity, we briefly review the theme of abelianization in Hamiltonian symplectic geometry. One begins with a result in representation theory: the finite-dimensional complex representations of a compact connected Lie group $G$ are completely determined by their restrictions to a maximal torus $T\subset G$. This can be viewed as a kind of ``abelianization phenomenon'' in representation theory, in which $G$ is replaced by the abelian group $T$. Guillemin--Jeffrey--Sjamaar \cite{gui-jef-sja:02} introduce \textit{symplectic implosion} as a symplecto-geometric analogue. Each Hamiltonian $G$-space $M$ thereby determines a stratified Hamiltonian $T$-space $M_{\mathrm{impl}}$, whose symplectic quotients by $T$ recover those of $M$ by $G$. One is justified in viewing the symplectic quotients of $M_{\mathrm{impl}}$ by $T$ as abelianizing those of $M$ by $G$. 
	
	Multiplicity-free branching suggests a refinement of this analogy, in which the maximal torus is replaced by a so-called ``Gelfand--Cetlin torus''. For $G=\operatorname{U}(n)$ and the usual diagonal maximal torus $T=\operatorname{U}(1)^n$, successive restriction along the standard upper-left chain $\operatorname{U}(1)\subset\cdots\subset\operatorname{U}(n)$ decomposes each irreducible representation $V$ into one-dimensional spaces indexed by integral Gelfand--Cetlin patterns \cite{molev2006Gelfand}. Giving each such space the character specified by its pattern defines an action of the abstract torus $\T_{\mathrm{big}}\coloneqq\operatorname{U}(1)^{n(n+1)/2}$ on the original vector space $V$. This action refines the $T$-weight decomposition of $V$. Given that restricting a $G$-representation to $T$ has a symplecto-geometric analogue, it is natural to seek a symplecto-geometric analogue of the $\T_{\mathrm{big}}$-action on $V$. Motivated by this viewpoint, the first author and Weitsman \cite{crooks2024Gelfand} show that the symplectic quotients of a Hamiltonian $G$-space at levels in a suitable open dense subset of $\g^*_{\mathrm{reg}}$ are canonically isomorphic to symplectic quotients of an open subset of that space by a ``Gelfand--Cetlin'' torus $\T_{\mathrm{big}}$. This result applies in arbitrary Lie type, extending the geometric abelianization suggested by the classical unitary branching construction. 
	
	\subsection{Subtleties involved in a generalization}\label{sec:intro-subtleties}
	The abelianization theorem of the first author and Weitsman requires the reduction level to lie in a suitable open dense subset of $$\g^*_{\mathrm{reg}}\coloneqq\{\xi\in\g^*\mid \dim\g_{\xi}=\rank G\}.$$ There are Hamiltonian $G$-spaces $M$ for which the moment map $\mu$ does not attain any values in $\g^*_{\mathrm{reg}}$; coadjoint orbits in $\g^*\setminus \g^*_{\mathrm{reg}}$ are prime examples. Another issue is that a connected Hamiltonian $G$-space $M$ with this property cannot satisfy the hypotheses of Marsden--Weinstein reduction at generic values of its moment map. This fact is witnessed by the non-trivial subgroup $[G_{\mu(m)},G_{\mu(m)}]\subset G$ acting trivially on $\mu^{-1}(\mu(m))$ for all $m$ in an open dense subset of $M$. The first author and Weitsman only consider reduction at levels for which this subgroup is trivial.
	
	\subsection{Summary of results}\label{sec:summary-results}
	Let $G$ be a compact connected Lie group with Lie algebra $\g$. We introduce \textit{generalized Gelfand--Cetlin data} for the orbit-type strata of $\g^*$, and use them to abelianize symplectic quotients at levels in the principal orbit-type stratum for a Hamiltonian $G$-space. This extends the framework of the first author and Weitsman \cite{crooks2024Gelfand} to spaces whose moment images may be disjoint from $\g^*_{\mathrm{reg}}$. 
	
	Fix an orbit-type stratum $\Sigma\subset\g^*$. Choose a fundamental Weyl chamber $\t_+^*$, put $\sigma\coloneqq\Sigma\cap\t_+^*$, and write $G_\sigma$ for the common stabilizer of the points of $\sigma$. Define
	\[
	\ell^{\Sigma}\coloneqq\dim\bigl(G_\sigma/[G_\sigma,G_\sigma]\bigr)=\dim\sigma,
	\quad
	u^{\Sigma}\coloneqq\frac12(\dim G-\dim G_\sigma),
	\quad\text{and}\quad
	b^{\Sigma}\coloneqq\ell^{\Sigma}+u^{\Sigma}.
	\]
	These dimensions determine the small, intermediate, and big tori
	\[
	\T_{\mathrm{small}}^{\Sigma}\coloneqq\operatorname{U}(1)^{\ell^{\Sigma}},
	\quad
	\T_{\mathrm{int}}^{\Sigma}\coloneqq\operatorname{U}(1)^{u^{\Sigma}},
	\quad\text{and}\quad
	\T_{\mathrm{big}}^{\Sigma}\coloneqq\T_{\mathrm{small}}^{\Sigma}\times\T_{\mathrm{int}}^{\Sigma}=\operatorname{U}(1)^{b^{\Sigma}}.
	\]
	We write $\R_{\mathrm{big}}^{\Sigma}\coloneqq\R^{b^{\Sigma}}$.
	
	A generalized Gelfand--Cetlin datum for $\Sigma$ consists of a continuous map $\lambda_{\mathrm{big}}^{\Sigma}:\g^*\longrightarrow\R_{\mathrm{big}}^{\Sigma}$ and an open dense subset $\Sigma_{\mathrm{reg}}\subset\Sigma$ satisfying Definition \ref{gen_gsdatum}. Its first $\ell^{\Sigma}$ components are $G$-invariant. The remaining $u^{\Sigma}$ components generate a Poisson Hamiltonian $\T_{\mathrm{int}}^{\Sigma}$-action on $\Sigma_{\mathrm{reg}}$, while $\T_{\mathrm{small}}^{\Sigma}$ acts trivially. These actions make the restriction
	\[
	\lambda_{\mathrm{big}}^{\Sigma}\big\vert_{\Sigma_{\mathrm{reg}}}:
	\Sigma_{\mathrm{reg}}\longrightarrow\R_{\mathrm{big}}^{\Sigma}
	\]
	a moment map for a Poisson Hamiltonian $\T_{\mathrm{big}}^{\Sigma}$-action. Viewed as a map onto its image, the same restriction is a principal $\T_{\mathrm{int}}^{\Sigma}$-bundle.
	
	The abelianization theorem of the first author and Weitsman \cite{crooks2024Gelfand} applies to stratified symplectic quotients at regular elements of $\g^*$. To extend it to other orbit-type strata, we use the toric degenerations and collective integrable systems of Hoffman--Lane \cite{hoffman2025stratified}. In Proposition \ref{prop:gc-existence}, we identify the effective torus on the piece of the degenerated universal implosion associated with $\Sigma$ and choose integral coordinates that incorporate $G_\sigma/[G_\sigma,G_\sigma]$. Sections \ref{construction_of_GS_datum} and \ref{Subsection: Construction} give criteria for constructing such a datum using the sweeping map and Thimm's method.
	
	Now let $(M,\omega,\mu)$ be a non-empty, connected Hamiltonian $G$-space. There exists a unique orbit-type stratum $\Sigma\subset\g^*$ such that $\mu(M)\cap\Sigma$ is non-empty and dense in $\mu(M)$. We call $\Sigma$ the \textit{principal orbit-type stratum} for the Hamiltonian $G$-space $M$. Given a generalized Gelfand--Cetlin datum for this stratum, write
	\[
	M_{\mathrm{reg}}\coloneqq\mu^{-1}(\Sigma_{\mathrm{reg}})\qquad\text{and}\qquad
	\lambda_M\coloneqq(\lambda_{\mathrm{big}}^{\Sigma}\circ\mu)\vert_{M_{\mathrm{reg}}}.
	\]
	The restriction $\lambda_M$ is a moment map for the induced Hamiltonian $\T_{\mathrm{big}}^{\Sigma}$-action on $M_{\mathrm{reg}}$.
	
	The following theorem extends the abelianization theorem of the first author and Weitsman \cite{crooks2024Gelfand} to arbitrary principal orbit-type strata.
	
	\begin{theorem*}
		Let $G$ be a compact connected Lie group with Lie algebra $\g$. Suppose that $\Sigma\subset\g^*$ is an orbit-type stratum.
		\begin{enumerate}
			\item[\textup{(i)}] There exists a generalized Gelfand--Cetlin datum for $\Sigma$.
			\item[\textup{(ii)}] Let $(\lambda_{\mathrm{big}}^{\Sigma},\Sigma_{\mathrm{reg}})$ be any generalized Gelfand--Cetlin datum for $\Sigma$. Suppose that $(M,\omega,\mu)$ is a non-empty, connected Hamiltonian $G$-space with principal orbit-type stratum $\Sigma$. For each $\xi\in\Sigma_{\mathrm{reg}}$, the inclusion of $\mu^{-1}(\xi)$ into $\lambda_M^{-1}(\lambda_{\mathrm{big}}^{\Sigma}(\xi))$ induces a canonical isomorphism
			\[
			M\sll{\xi}G\simeq M_{\mathrm{reg}}\sll{\lambda_{\mathrm{big}}^{\Sigma}(\xi)}\T_{\mathrm{big}}^{\Sigma}
			\]
			of stratified symplectic spaces.
		\end{enumerate}
	\end{theorem*}
	
	For $\xi\in\Sigma_{\mathrm{reg}}$, the subgroup $[G_\xi,G_\xi]$ fixes $\mu^{-1}(\xi)$ pointwise, and the chosen datum identifies $G_\xi/[G_\xi,G_\xi]$ with $\T_{\mathrm{small}}^\Sigma$. Under this identification, a point of $\mu^{-1}(\xi)$ with $G_\xi$-stabilizer $H$ has $\T_{\mathrm{big}}^\Sigma$-stabilizer $(H/[G_\xi,G_\xi])\times\{e\}$. This is the correspondence between orbit types underlying the isomorphism in part \textup{(ii)}.
	
	Section \ref{sec:moment_map_image_submanifold} assigns to each orbit-type stratum $\Sigma\subset\g^*$ a Hamiltonian $G$-space $M$ with $\Sigma$ as its moment image. For $G=\operatorname{U}(n)$ with $n\geq3$, Example \ref{ex:singular-u3-datum} constructs a generalized Gelfand--Cetlin datum for the stratum with ordered eigenvalue multiplicities $(n-1,1)$. Every point of $\Sigma$ is a non-regular element of $\mathfrak u(n)^*$. For each $\xi\in\Sigma_{\mathrm{reg}}$, the reduction $M\sll{\xi}\operatorname{U}(n)$ is diffeomorphic to $\mathbb{CP}^{n-1}$. Part \textup{(ii)} identifies this quotient with a symplectic reduction by $\operatorname{U}(1)^{n+1}$.
	
	\subsection*{Acknowledgements} P.C. was supported by the National Science Foundation grant DMS-2454103 and Simons Foundation grant MP-TSM-00002292.
	
	\subsection*{AI declaration}
	The authors used ChatGPT Astra 6.2 to help interpret the notation and conventions of Hoffman--Lane~\cite{hoffman2025stratified}, and to help develop Example~\ref{ex:singular-u3-datum}. We reviewed all AI output, wrote this manuscript, and take full responsibility for its contents.
	
	%%%%%%%%%%%%%%%%%%%%%%%%%%%%%%%%%%%%%%%%%%%%%%%%%%%%%%%%%%%%%%%%%%%
	%%%%%%%%%%%%%%%%%%%%%%%  SECTION 2  %%%%%%%%%%%%%%%%%%%%%%%%%%%%%%%
	%%%%%%%%%%%%%%%%%%%%%%%%%%%%%%%%%%%%%%%%%%%%%%%%%%%%%%%%%%%%%%%%%%%
	\section{Preliminaries and conventions}
	We now collect some of the Lie-theoretic and Poisson-geometric machinery underlying this work. In Section \ref{sec:lie-theory}, we establish relevant conventions on compact connected Lie groups. This leads to Section \ref{sec:Unitary_and_tori}, where we address specific conventions for unitary groups and tori. We review the orbit-type stratification for a compact Lie group action in Section \ref{group_actions_on_manifold}. In Section \ref{Subsection: Hamiltonian geometry}, we review Poisson Hamiltonian actions.
	
	\subsection{Compact connected Lie groups}\label{sec:lie-theory}
	Let $G$ be a compact connected Lie group with Lie algebra $\g$ and rank $r$. 
	Consider the adjoint representations 
	\begin{align*}
		\Ad:G \times \g \longrightarrow \g \quad\text{and}\quad \ad:\g \times \g \longrightarrow \g,
	\end{align*}
	as well as the corresponding coadjoint representations
	\begin{align*}
		\Ad^*:G \times \g^* \longrightarrow \g^* \quad\text{and}\quad \ad^*:\g \times \g^* \longrightarrow \g^*.
	\end{align*}
	For $x\in\g$ and $\xi\in\g^*$, we have the $G$-stabilizers $$G_x\coloneqq\{g\in G\mid \Ad_g(x)=x\}\quad\text{and}\quad G_{\xi}\coloneqq\{g\in G\mid \Ad_g^*(\xi)=\xi\},$$ and $\g$-centralizers
	$$\g_x\coloneqq\{y\in\g\mid \ad_y(x)=0\}\quad\text{and}\quad\g_{\xi}\coloneqq\{y\in\g\mid \ad_y^*(\xi)=0\}.$$
	
	Choose a maximal torus $T \subset G$ with Lie algebra $\t \subset \g$. 
	Consider the complexifications $\g_\C \coloneqq \g\otimes_\R \C$ and $\t_\C \coloneqq \t\otimes_\R \C$.
	Write $\Phi \subset \t_\C^*$ for the roots of $\g_\C$ with respect to $\t_\C$, and choose a collection $\Phi^+ \subset \Phi$ of positive roots.
	This choice gives rise to the fundamental Weyl chamber
	$$
	\t_+ \coloneqq \{ x \in \t \mid -i\alpha(x) \geq 0 \ \forall \alpha \in \Phi^+ \}.
	$$
	It is known to be a polyhedral cone and fundamental domain for the adjoint action of $G$. 
	
	As $G$ is compact, we can choose an inner product $\langle -,- \rangle:\g\otimes_{\R}\g\longrightarrow\R$ on $\g$ that is invariant under the adjoint action. 
	It follows that 
	\begin{equation}\label{innerprod_iso}
		\begin{aligned}
			( - )^\vee: \g &\longrightarrow \g^*\\
			x &\longmapsto \langle x, - \rangle 
		\end{aligned}
	\end{equation}
	is a $G$-module isomorphism.
	Using the inner product, we identify $\t^*$ with $\operatorname{ann}_{\g^*}(\t^\perp)\subset\g^*$. The dual fundamental Weyl chamber is 
	$$\t_+^* \coloneqq (-)^\vee (\t_+) \subset \t^* \subset \g^*.$$
	This polyhedral cone is a fundamental domain for the coadjoint action of $G$.
	
	The \textit{regular loci} are defined by
	\[
	\g_{\mathrm{reg}} \coloneqq \{ x \in \g \mid \dim \g_x = r \} \qquad \text{and} \qquad 
	\g^*_{\mathrm{reg}} \coloneqq \{ \xi \in \g^* \mid \dim \g_\xi = r \}.
	\]
	Note that $\dim \g_x \geq r$ and $\dim \g_\xi \geq r$ for all $x \in \g$ and all $\xi \in \g^*$. 
	An element $x \in \g$ (resp. $\xi \in \g^*$) is in $\g_{\mathrm{reg}}$ (resp. $\g^*_{\mathrm{reg}}$) if and only if $\g_x$ (resp. $\g_\xi$) is a Cartan subalgebra of $\g$. 
	Equivalently, $x \in \g$ (resp. $\xi \in \g^*$) is a regular element if and only if $G_x$ (resp. $G_\xi$) is a maximal torus of $G$. 
	\subsection{Unitary groups and tori}\label{sec:Unitary_and_tori}
	Let $n$ be a non-negative integer. 
	One especially important compact connected Lie group is the unitary group 
	$$
	\operatorname{U}(n) \coloneqq \{A \in \operatorname{GL}_n \mid A^{*}A = I_n\}.
	$$
	Its Lie algebra $\mathfrak{u}(n)$ is the real vector space of skew-Hermitian $n\times n$ matrices, equipped with the commutator bracket.
	On the other hand, let $\mathcal{H}(n)$ denote the real vector space of $n \times n$ Hermitian matrices.
	Conjugation by elements of $\operatorname{U}(n)$ renders $\mathcal{H}(n)$ a real $\operatorname{U}(n)$-module.
	The bilinear pairing
	\begin{align*}
		\mathfrak{u}(n) \otimes_\R \mathcal{H}(n) &\longrightarrow \R\\
		x \otimes \xi &\longmapsto -i\operatorname{tr}(x \xi)
	\end{align*}
	is non-degenerate and $\operatorname{U}(n)$-invariant.
	It therefore identifies $\mathfrak{u}(n)^*$ and $\mathcal{H}(n)$ as $\operatorname{U}(n)$-modules. 
	In the special case where $n=1$, we have $\mathfrak{u}(1) = i\R$ and $\mathfrak{u}(1)^* =\mathcal{H}(1)=\R$. 
	
	Throughout this paper, we define a torus to be a compact connected abelian Lie group. 
	As such, it is isomorphic to $\operatorname{U}(1)^k$ for some $k \in \Z_{\geq 0}$. 
	If $T = \operatorname{U}(1)^k$, its Lie algebra is $\t= \mathfrak{u}(1)^k = (i\R)^k$ with dual $\t^* = (\mathfrak{u}(1)^*)^k = \R^k$.
	
	For a torus $T$ with Lie algebra $\t$, we use the integral lattice
	\[
	\Lambda_{\t}\coloneqq\frac{1}{2\pi}
	\ker\bigl(\exp:\t\longrightarrow T\bigr).
	\]
	This lattice has rank $\dim T$. We use the same normalization for
	subtori and quotient tori.
	\subsection{Group actions on manifolds}\label{group_actions_on_manifold}
	Let $G$ be a Lie group with Lie algebra $\g$.
	Suppose that $M$ is a smooth manifold equipped with a smooth (left) action of $G$.
	The generating vector field of $x \in \g$ is defined by
	$$
	x_M(p) \coloneqq \frac{\mathrm{d}}{\mathrm{d}t}\bigg\vert_{t=0} (\exp(-tx)\cdot p)
	$$
	for all $p \in M$. The stabilizer subgroup for $p \in M$ is
	$$
	G_p \coloneqq \{ g \in G \mid g \cdot p = p \}.
	$$
	For any closed subgroup $H \subset G$, we define the subset
	$$
	M_H \coloneqq \{p \in M \mid H = G_p\}.$$
	
	\begin{definition}\label{points_of_orbit_type}
		Let $\mathcal{C}$ be a conjugacy class of closed subgroups of $G$.
		The set of points in $M$ of \textit{orbit-type} $\mathcal{C}$ is defined by
		$$
		M_\mathcal{C}\coloneqq \{ p \in M \mid G_p \in \mathcal{C} \}.
		$$
		We say that $\mathcal{C}$ is an \textit{orbit-type of} $M$ if $M_\mathcal{C} \neq \emptyset$.
	\end{definition}
	
	For a closed subgroup $H\subset G$, write $(H)$ for its conjugacy
	class. Then $M_{(H)}=G\cdot M_H$.
	
	\subsection{Hamiltonian actions and symplectic reduction}\label{Subsection: Hamiltonian geometry}
	
	Let $(M, \omega)$ be a symplectic manifold with sheaf $\mathcal{O}_M$ of smooth real-valued functions. The Hamiltonian vector field of $f\in C^\infty(M)$ is given by $X_f \coloneqq -(\omega^\vee)^{-1}(\mathrm{d}f)$, where $\omega^{\vee}:TM\longrightarrow T^*M$ is the vector bundle isomorphism obtained by contracting $\omega$ with tangent vectors in the first argument. Suppose also that $M$ is equipped with a smooth action of a Lie group $G$ with Lie algebra $\g$.
	For a map $\mu:M\longrightarrow \g^*$ and $x\in\g$, write $\mu^x(p)\coloneqq\mu(p)(x)$ for $p\in M$.
	One calls $(M,\omega,\mu)$ a \textit{Hamiltonian $G$-space} if $\omega$ is $G$-invariant, $\mu$ is smooth and $G$-equivariant, and $X_{\mu^x}=x_M$ for all $x \in \g$.
	In this case, the map $\mu:M \longrightarrow \g^*$ is called the \textit{moment map}. If a restriction of $\mu$ takes values in a submanifold $S\subset\g^*$, we also write $\mu$ for the resulting map to $S$, with the inclusion $S\hookrightarrow\g^*$ understood in compositions.
	
	Next consider a smooth manifold $M$ with a bivector field $\sigma\in\Gamma(M,\wedge^2TM)$.
	One calls $(M, \sigma)$ a Poisson manifold if 
	$$
	\{-,-\}: \mathcal{O}_M \times \mathcal{O}_M \longrightarrow \mathcal{O}_M, \quad \{f,g\} \coloneqq \sigma(\mathrm{d}f, \mathrm{d}g)
	$$
	renders the structure sheaf $\mathcal{O}_M$ a sheaf of Poisson algebras.
	Assume this to be the case.
	We may contract $\sigma$ with cotangent vectors in the first argument to obtain the vector bundle morphism $\sigma^\vee:T^*M \longrightarrow TM$. 
	If $\sigma^\vee$ is a vector bundle isomorphism, then we call $\sigma$ non-degenerate. 
	In this case, there is a unique symplectic form $\omega \in \Omega^2(M)$ such that $(\sigma^\vee)^{-1} = -\omega^\vee$. 
	This gives a bijection between symplectic structures on $M$ and non-degenerate Poisson structures on $M$. 
	
	For a Poisson manifold $(M, \sigma)$, the Hamiltonian vector field of a smooth function $f \in C^\infty(M)$ is defined by $X_f \coloneqq \sigma^\vee(\mathrm{d}f)$.
	This generalizes the above definition of a Hamiltonian vector field on a symplectic manifold.
	A Poisson Hamiltonian $G$-space is then defined in a similar manner to the symplectic case. 
	In more detail, suppose that a Lie group $G$ with Lie algebra $\g$ acts smoothly on a Poisson manifold $(M, \sigma)$.
	Let $\mu:M\longrightarrow \g^*$ be a smooth map.
	One calls $(M, \sigma,\mu)$ a \textit{Poisson Hamiltonian $G$-space} if $\sigma$ is $G$-invariant, $\mu$ is $G$-equivariant, and $X_{\mu^x} = x_M$ for all $x \in \g$.
	As above, $\mu$ is called the \textit{moment map}.
	Note that a Hamiltonian $G$-space is simply a Poisson Hamiltonian $G$-space for which the Poisson structure is non-degenerate. 
	
	\begin{example}
		Let $G$ be a Lie group with Lie algebra $\g$. The dual $\g^*$ has a canonical Poisson Hamiltonian $G$-space structure. 
		Given $f_1, f_2 \in C^\infty(\g^*)$, the Poisson bracket is defined as
		$$
		\{f_1, f_2\}(\xi) \coloneqq \xi\big([(\mathrm{d} f_1)_\xi, (\mathrm{d} f_2)_\xi]\big).
		$$
		The coadjoint action of $G$ is Poisson Hamiltonian with moment map the identity $\g^* \longrightarrow \g^*$. 
	\end{example}
	
	We now briefly review Marsden--Weinstein reduction for Hamiltonian $G$-spaces \cite{marsden1974Reduction}. The interested reader is referred to \cite{mar-rat:86} for a generalization to Poisson Hamiltonian $G$-spaces.
	
	\begin{theorem}[Marsden--Weinstein]\label{Theorem: MW}
		Let $G$ be a Lie group with Lie algebra $\g$. Suppose that $(M,\omega,\mu)$ is a Hamiltonian $G$-space. Fix a point $\xi\in\g^*$, and write $\pi:\mu^{-1}(\xi)\longrightarrow\mu^{-1}(\xi)/G_{\xi}$ and $j:\mu^{-1}(\xi)\longrightarrow M$ for the quotient and inclusion, respectively. 
		\begin{itemize}
			\item[\textup{(i)}] If the action of $G_{\xi}$ on $\mu^{-1}(\xi)$ is free, then $\mu^{-1}(\xi)$ is a submanifold of $M$.
			\item[\textup{(ii)}] If the action of $G_{\xi}$ on $\mu^{-1}(\xi)$ is free and proper, then the quotient topological space $\mu^{-1}(\xi)/G_{\xi}$ inherits a unique symplectic manifold structure $(\mu^{-1}(\xi)/G_{\xi},\omega_{\xi})$ for which $\pi$ is a submersion and $\pi^*\omega_{\xi}=j^*\omega$.
		\end{itemize}
	\end{theorem}
	
	Under the freeness and properness assumptions of Theorem \ref{Theorem: MW}, we set
	$$M\sll{\xi}G\coloneqq (\mu^{-1}(\xi)/G_{\xi},\omega_{\xi}).$$
	%%%%%%%%%%%%%%%%%%%%%%%%%%%%%%%%%%%%%%%%%%%%%%%%%%%%%%%%%%%%%%%%%%%
	%%%%%%%%%%%%%%%%%%%%%%%  SECTION 3  %%%%%%%%%%%%%%%%%%%%%%%%%%%%%%%
	%%%%%%%%%%%%%%%%%%%%%%%%%%%%%%%%%%%%%%%%%%%%%%%%%%%%%%%%%%%%%%%%%%%
	
	\section{Orbit-type stratifications and symplectic reduction}\label{sec:stratifications}
	Consider a compact connected Lie group $G$ with Lie algebra $\mathfrak{g}$, a Hamiltonian $G$-space $(M,\omega,\mu)$, and an element $\xi \in \g^*$. 
	Suppose that the action of $G_\xi$ on $\mu^{-1}(\xi)$ is not free. 
	One finds that $\mu^{-1}(\xi)$ need not be a submanifold of $M$.
	Even in cases where $\mu^{-1}(\xi)$ is a submanifold of $M$, the failure of freeness means that $\mu^{-1}(\xi)/G_\xi$ need not admit a smooth manifold structure for which the quotient map $\mu^{-1}(\xi) \longrightarrow \mu^{-1}(\xi)/G_\xi$ is a submersion.
	One would still like to define $M\sll{\xi}G$ when the usual hypothesis of Marsden--Weinstein reduction may not be satisfied.
	Sjamaar--Lerman \cite{sjamaar1991Stratified} accomplish this via their framework of \textit{stratified symplectic spaces}.
	Their approach is to stratify $\mu^{-1}(\xi)/G_\xi$ into a disjoint union of symplectic manifolds that can be studied independently. 
	In what follows, we briefly review their constructions and relevant background material.
	We use the Sjamaar--Lerman structures on the reduced spaces
	\cite{sjamaar1991Stratified}. Our notion of an isomorphism between them
	adapts \cite[Definition~13]{crooks2024Gelfand} to possibly non-abelian
	stabilizers of the reduction levels.
	
	In Section \ref{sec:orbit-type-strata-section}, we consider the orbit-type stratification of a general $G$-manifold \cite{duistermaat}. The special case of the coadjoint representation is considered in Section \ref{Subsection: The case}.
	Section \ref{sec:orbit-type-and-reduction-section} then addresses how the orbit-type stratification relates to symplectic reduction.
	This is further developed in Section \ref{sec:principal-stratum-section}, where we connect the orbit-type stratification of $\g^*$ to the geometry of moment maps.

	\subsection{Orbit-type strata}\label{sec:orbit-type-strata-section}
	We use stratified spaces in the sense of Sjamaar and Lerman \cite{sjamaar1991Stratified}.
	Consider a smooth $G$-manifold $M$ for a compact connected Lie group $G$.
	Recall Definition \ref{points_of_orbit_type} from Section \ref{group_actions_on_manifold}. 
	Let $\{M_\mathcal{C}^i\}_{i\in I_\mathcal{C}}$ denote the set of connected components of $M_\mathcal{C}$ for an orbit-type $\mathcal{C}$ of $M$.
	Note that $M$ is the disjoint union
	\begin{equation}\label{orbit-type-decomp}
		M = \bigsqcup_{\mathcal{C},i} M_{\mathcal{C}}^i.
	\end{equation}
	Each component $M_{\mathcal{C}}^i$ is invariant under the $G$-action.
	This fact gives rise to the disjoint union
	\begin{equation}\label{quot-orbit-type-decomp}
		M/G = \bigsqcup_{\mathcal{C},i} M_{\mathcal{C}}^i\big/G.
	\end{equation}
	Equations \ref{orbit-type-decomp} and \ref{quot-orbit-type-decomp} give $M$ and $M\slash G$ the structures of stratified spaces called the \textit{orbit-type stratifications} of $M$ and $M/G$ \cite{bredon1972introduction}.
	
	Assume that $M$ is non-empty and connected.
	Among the orbit-types of $M$, there is a unique class $\mathcal{C}_{\mathrm{prin}}$ with the following property:
	for all orbit-types $\mathcal{C}$ of $M$ and elements $H_{\mathrm{prin}} \in \mathcal{C}_{\mathrm{prin}}$, there exists $H \in \mathcal{C}$ such that $H_{\mathrm{prin}} \subset H$. 
	The corresponding orbit-type stratum $M_{\mathrm{prin}}\coloneqq M_{\mathcal{C}_{\mathrm{prin}}}$ is open and dense in $M$, and its quotient $M_{\mathrm{prin}}/G \subset M \big/G$ is open and dense.
	The stratum $M_{\mathrm{prin}} \subset M$ is called the \textit{principal stratum} for the action of $G$ on $M$.
	
	\subsection{Orbit-type strata for the coadjoint action}\label{Subsection: The case}
	Let $G$ be a compact connected Lie group with Lie algebra $\g$. 
	Consider the coadjoint action of $G$ on $\g^*$.
	In this context, the orbit-type strata are especially well understood. 
	With the choice of a fundamental Weyl chamber $\t_+^* \subset \g^*$, the open faces of $\t_+^*$ correspond to the orbit-type strata as follows:
	\begin{align*}
		\{\text{Orbit-type strata of } \g^*\} &\longleftrightarrow \{\text{Open faces of }\t_+^*\}\\
		G \cdot \tau &\longmapsfrom \tau \\
		\Sigma &\longmapsto \Sigma \cap \t_+^*.
	\end{align*}
	In this correspondence, a stratum is a connected component of an orbit-type set. For example, in $\operatorname{U}(3)$ the ordered eigenvalue multiplicities $(2,1)$ and $(1,2)$ give conjugate stabilizers and distinct strata.
	The principal stratum of the coadjoint action is the one corresponding to the interior of the Weyl chamber in $\t^*$, so that
	$$
	\g^*_{\mathrm{reg}} = G \cdot (\t_+^*)^\circ.
	$$
	Note that this is the regular locus of $\g^*$, as defined in Section \ref{sec:lie-theory}.
	
	\begin{lemma}\label{strata_cotangent_space}
		Consider an orbit-type stratum $\Sigma \subset \g^*$.
		For all $\xi \in \Sigma$, we have
		$$
		T_\xi\Sigma = \operatorname{ann}_{\g^*}([\g_\xi,\g_\xi]).
		$$
	\end{lemma}
	\begin{proof}
		
		Consider the $G$-module isomorphism $(-)^\vee:\g\longrightarrow \g^*$ from Equation \ref{innerprod_iso}.
		Let $x \in \g$ be the element such that $x^\vee=\xi$. Both sides of the claimed equality are $G$-equivariant, so after conjugation we may assume that $x$ lies in the corresponding open face of $\t_+$.
		Recall that we have the orthogonal decompositions
		\begin{equation}\label{lie_alg_decomp}
			\g = \g_x \oplus [\g, x] = \mathfrak{z}(\g_x)  \oplus [\g_x, \g_x] \oplus [\g, x].
		\end{equation}
		Let $\Sigma_\g \subset \g$ be the orbit-type stratum of the adjoint representation containing $x$, and write $\sigma_{\g}\subset\mathfrak{t}_{+}$ for the open face to which it corresponds.
		We have
		\begin{equation}\label{lie_alg_stratum_tangent_decomp}
			T_x \Sigma_\g = T_x(G\cdot x) \oplus T_x \sigma_\g = [\g,x] \oplus \mathfrak{z}(\g_x).
		\end{equation}
		It follows that $T_x \Sigma_\g = [\g_x, \g_x]^\perp$.
		We have $(-)^\vee(\Sigma_\g) = \Sigma$, and therefore $T_\xi\Sigma = \operatorname{ann}_{\g^*}([\g_\xi, \g_\xi])$.
	\end{proof}
	
	\subsection{Orbit-type strata and symplectic reduction}\label{sec:orbit-type-and-reduction-section}
	The orbit-type stratification plays an important role in the Sjamaar--Lerman approach to symplectic reduction. 
	Let $G$ be a compact connected Lie group with Lie algebra $\g$, and suppose that $(M,\omega,\mu)$ is a Hamiltonian $G$-space.
	Recall from the introduction to Section \ref{sec:stratifications} the precise senses in which $\mu^{-1}(\xi)/ G_\xi$ may fail to be a manifold, where $\xi \in \g^*$. 
	On the other hand, $\mu^{-1}(\xi) / G_\xi$ is always a \textit{stratified symplectic space} \cite{sjamaar1991Stratified}. 
	We now review some of the pertinent ideas.
	
	Let $G$ be a compact connected Lie group with Lie algebra $\g$. 
	Suppose that $(M,\omega,\mu)$ is a Hamiltonian $G$-space.
	For $\xi \in \g^*$, one aims to realize $\mu^{-1}(\xi) / G_\xi$ as a stratified symplectic space. A first step is to examine the orbit-type stratifications arising from the action of $G_\xi$ on $\mu^{-1}(\xi)$.
	These stratifications are 
	$$
	\mu^{-1}(\xi) = \bigsqcup_{\mathcal{C},i}\mu^{-1}(\xi)_{\mathcal{C}}^i \quad \text{and} \quad 
	\mu^{-1}(\xi)/G_\xi = \bigsqcup_{\mathcal{C},i}\big(\mu^{-1}(\xi)_{\mathcal{C}}^i /G_\xi\big)
	$$
	where $\mathcal{C}$ ranges over orbit-types of $\mu^{-1}(\xi)$; see Equations \ref{orbit-type-decomp} and \ref{quot-orbit-type-decomp}. While $\mu^{-1}(\xi)$ may fail to be a manifold, the definitions of $\mu^{-1}(\xi)_{\mathcal{C}}$ and $\mu^{-1}(\xi)_{\mathcal{C}}^i$ still make sense.
	
	\begin{theorem}[Sjamaar--Lerman]\label{thrm:Sjamaar-Lerman}
		Suppose that $\xi \in \g^*$, and let $\mathcal{C}$ be an orbit-type of $\mu^{-1}(\xi)$.
		Let $\{\mu^{-1}(\xi)_\mathcal{C}^i\}_{i \in I_{\mathcal C}}$ be the connected components of $\mu^{-1}(\xi)_\mathcal{C}$.
		\begin{enumerate}[(i)]
			\item The subset $\mu^{-1}(\xi)_{\mathcal{C}}^i$ is a submanifold of $M$ for all $i \in I_{\mathcal C}$.
			\item For all $i \in I_{\mathcal C}$, there exists a unique smooth manifold structure on $\mu^{-1}(\xi)_\mathcal{C}^i/G_\xi$ for which the quotient map $\pi:\mu^{-1}(\xi)_\mathcal{C}^i\longrightarrow \mu^{-1}(\xi)_\mathcal{C}^i/G_\xi$ is a submersion.
			\item There exists a unique symplectic form $(\omega_\xi)_{\mathcal{C}}^i$ on $\mu^{-1}(\xi)_{\mathcal{C}}^i/G_\xi$ satisfying 
			$$
			\pi^*\bigl((\omega_\xi)_{\mathcal{C}}^i\bigr) = j^*\omega
			$$
			for all $i \in I_{\mathcal C}$, where $j: \mu^{-1}(\xi)_\mathcal{C}^i \longrightarrow M$ is the inclusion. 
		\end{enumerate}
	\end{theorem}
	
	Sjamaar--Lerman also prove that $\mu^{-1}(\xi)/G_\xi$ is a stratified symplectic space with strata $\{\mu^{-1}(\xi)_\mathcal{C}^i/G_\xi\}_{\mathcal{C},i}$.
	We denote this stratified symplectic space as
	$M\sll{\xi}G$.
	
	Given $\xi\in\g^*$, consider the quotient 
	$$\Stab_M(G,\xi)\coloneqq \{(G_\xi)_p \mid p \in \mu^{-1}(\xi)\} \ /\, G_\xi$$ by the conjugation action of $G_\xi$. 
	This is the set of $G_\xi$-conjugacy classes of stabilizers of the $G_\xi$-action on $\mu^{-1}(\xi)$.
	
	\begin{definition}\label{sss_isomorph}
		Let $G$ and $K$ be compact connected Lie groups with respective Lie algebras $\g$ and $\mathfrak k$. 
		Suppose that $(M,\omega,\mu)$ is a Hamiltonian $G$-space and $(N,\omega_N,\nu)$ is a Hamiltonian $K$-space.
		Take $\xi \in \g^*$ and $\eta \in \mathfrak k^*$.
		A pair of maps 
		\begin{align*}
			\varphi:M \sll{\xi} G \longrightarrow N \sll{\eta} K \quad \text{and} \quad 
			\phi: \Stab_M(G, \xi) \longrightarrow \Stab_N(K, \eta)
		\end{align*}
		is called an \textit{isomorphism of stratified symplectic spaces} if the following conditions are satisfied:
		\begin{enumerate}[(i)]
			\item $\varphi$ is a homeomorphism;
			\item $\phi$ is a bijection;
			\item $\varphi$ restricts to a symplectomorphism 
			\[
			(\mu^{-1}(\xi)_{\mathcal{C}})\big/ G_\xi \longrightarrow (\nu^{-1}(\eta)_{\phi(\mathcal{C})})\big/ K_\eta 
			\]
			for all $\mathcal{C}\in \Stab_M(G, \xi)$, with the symplectic structures and the symplectomorphism understood componentwise.
		\end{enumerate}
	\end{definition}
	
	\subsection{The principal orbit-type stratum for a Hamiltonian \texorpdfstring{$G$}{G}-space}\label{sec:principal-stratum-section}
	
	Let us continue with the conventions and notation from the previous section.
	As the moment map $\mu:M \longrightarrow \g^*$ is $G$-equivariant, it is natural to ask about the relationship between orbit-type strata in $M$ and $\g^*$. 
	The symplectic cross-section theorem \cite{lerman1998Nonabelian} implies the following result of this nature. 
	
	\begin{theorem}\label{sympl_cross_section}
		Let $G$ be a compact connected Lie group and $(M,\omega,\mu)$ a non-empty, connected Hamiltonian $G$-space.  
		There exists a unique orbit-type stratum $\Sigma \subset \g^*$ with the property that $\mu(M) \cap \Sigma$ is non-empty and dense in $\mu(M)$.      
		Furthermore, $\mu^{-1}(\Sigma) \subset M$ is open and dense. 
	\end{theorem}
	
	\begin{definition}\label{def:M-principal-stratum}
		We call the orbit-type stratum $\Sigma\subset\g^*$ in Theorem \ref{sympl_cross_section} the \textit{principal orbit-type stratum} for the Hamiltonian $G$-space $M$. The notation $M_{\mathrm{prin}}\subset M$ in Section \ref{sec:orbit-type-strata-section} refers to the principal stratum for the action of $G$ on $M$ itself.
	\end{definition}
	
	\begin{proposition} \label{trivial_action}
		Let $(M,\omega,\mu)$ be a non-empty, connected Hamiltonian $G$-space with principal orbit-type stratum $\Sigma \subset \g^*$.
		For $\xi \in \Sigma$, the action of $[G_\xi, G_\xi] \subset G_\xi$ on $\mu^{-1}(\xi)$ is trivial.
	\end{proposition}
	\begin{proof}
		This is proven in \cite{lerman1998Nonabelian} for $\xi$ in the principal face $\sigma\subset\mathfrak{t}_{+}^*$ determined by $M$. Since $\Sigma=G\cdot\sigma$ and $\mu$ is $G$-equivariant, the desired result follows.
	\end{proof}

	%%%%%%%%%%%%%%%%%%%%%%%%%%%%%%%%%%%%%%%%%%%%%%%%%%%%%%%%%%%%%%%%%%%
	%%%%%%%%%%%%%%%%%%%%%%%  SECTION 4  %%%%%%%%%%%%%%%%%%%%%%%%%%%%%%%
	%%%%%%%%%%%%%%%%%%%%%%%%%%%%%%%%%%%%%%%%%%%%%%%%%%%%%%%%%%%%%%%%%%%
	
	\section{Generalized Gelfand--Cetlin data}\label{GS_data_section}
	We now define and examine \textit{generalized Gelfand--Cetlin data} in the context of symplectic reduction. In Section \ref{sec:CW-gelfand-cetlin-datum}, we review results of the first author and Weitsman on Gelfand--Cetlin data and abelianized symplectic quotients. The definition of a generalized Gelfand--Cetlin datum then appears in Section \ref{sec:gen-GC-data-definition}. We show that such data exist in Section \ref{Subsection: Existence}. Section \ref{construction_of_GS_datum} constructs the invariant components from the sweeping map. Section \ref{Subsection: Construction} gives a criterion for completing them to a generalized Gelfand--Cetlin datum using Thimm's method.
	
	\subsection{Gelfand--Cetlin data in the regular case}\label{sec:CW-gelfand-cetlin-datum}
	We first recall the terminology of the first author and Weitsman \cite{crooks2024Gelfand}. Let $G$ be a compact connected Lie group with Lie algebra $\g$ and rank $r$. Define the quantities
	$$u \coloneqq \frac{1}{2}(\dim \g - r) \qquad \text{and}\qquad b\coloneqq \frac{1}{2}(\dim \g + r)$$ and tori
	$$\T_{\mathrm{int}}\coloneqq\operatorname{U}(1)^u\quad\text{and}\quad\T_{\mathrm{big}}\coloneqq\operatorname{U}(1)^{b}.$$ We identify $\T_{\mathrm{int}}$ with the product of the last $u$ factors of $\T_{\mathrm{big}}$.
	
	\begin{definition}\label{gsdatum}
		A \textit{Gelfand--Cetlin datum} is a pair $(\lambda_{\mathrm{big}}, \g^{*}_{\mathrm{s-reg}})$ of a continuous map
		$$\lambda_{\mathrm{big}}\coloneqq(\lambda_{1}, \dots, \lambda_{b}): \g^{*} \longrightarrow \R^b$$ and open dense subset $\g^{*}_{\mathrm{s-reg}} \subset \g^{*}$ that satisfy the following conditions:
		\begin{itemize}
			\item[\textup{(i)}] The first $r$ component functions $\lambda_{1}, \dots, \lambda_{r}:\g^* \longrightarrow \R$ are $G$-invariant on $\g^{*}$ and smooth on $\g^{*}_{\mathrm{reg}}$;
			\item[\textup{(ii)}] $\{(\mathrm{d}\lambda_{1})_{\xi}, \dots, (\mathrm{d}\lambda_{r})_{\xi}\}$ is a $\Z$-basis of the integral lattice $\Lambda_{\g_{\xi}}\subset \g_{\xi}$ for all $\xi \in \g^{*}_{\mathrm{reg}}$;
			\item[\textup{(iii)}] $\g^{*}_{\mathrm{s-reg}} \subset \g^{*}_{\mathrm{reg}}$;
			\item[\textup{(iv)}] $\lambda_{\mathrm{big}}\big\vert_{\g^{*}_{\mathrm{s-reg}}}:\g^*_{\mathrm{s-reg}} \longrightarrow \R^{b}$ is a smooth submersion and moment map for a Poisson Hamiltonian $\T_{\mathrm{big}}$-space structure on $\g^{*}_{\mathrm{s-reg}}$;
			\item[\textup{(v)}] the restriction $\lambda_{\mathrm{big}}\big\vert_{\g^{*}_{\mathrm{s-reg}}}:\g^{*}_{\mathrm{s-reg}} \longrightarrow \lambda_{\mathrm{big}}(\g^{*}_{\mathrm{s-reg}})$ is a principal $\T_{\mathrm{int}}$-bundle for the restriction of the action in Condition \textup{(iv)};
			\item[\textup{(vi)}] for any Hamiltonian $G$-space $(M,\omega,\mu)$, $\mu^{-1}(\g^*_{\mathrm{s-reg}})$ is a Hamiltonian $\T_{\mathrm{big}}$-space with moment map 
			$$(\lambda_{\mathrm{big}} \circ \mu)\big\vert_{\mu^{-1}(\g^{*}_{\mathrm{s-reg}})}:\mu^{-1}(\g^{*}_{\mathrm{s-reg}}) \longrightarrow \R^{b}.$$
		\end{itemize}
	\end{definition}
	
	In \cite{crooks2024Gelfand}, the first author and Weitsman prove the following.
	
	\begin{theorem}
		Suppose that $(\lambda_{\mathrm{big}},\g^*_{\mathrm{s-reg}})$ is a Gelfand--Cetlin datum. Given a Hamiltonian $G$-space $(M,\omega,\mu)$ and an element $\xi\in\g^*_{\mathrm{s-reg}}$, there is a canonical isomorphism of stratified symplectic spaces
		$$
		M\sll{\xi}G\simeq
		\mu^{-1}(\g^*_{\mathrm{s-reg}})
		\sll{\lambda_{\mathrm{big}}(\xi)}\T_{\mathrm{big}}.
		$$
	\end{theorem}
	
	It is important to note that the preceding theorem only applies to elements $\xi \in \g^*_{\mathrm{s-reg}}$, which are necessarily contained in $\g^*_{\mathrm{reg}}$. 
	For moment maps that do not achieve values in $\g^*_{\mathrm{reg}}$, this theorem is not particularly meaningful.
	
	\subsection{Generalized Gelfand--Cetlin data}\label{sec:gen-GC-data-definition}
	Let $G$ be a compact connected Lie group with Lie algebra $\g$, and fix an orbit-type stratum $\Sigma\subset\g^*$. Choose a fundamental Weyl chamber $\t_+^*\subset\g^*$, and let $\sigma\coloneqq\Sigma\cap\t_+^*$ be the corresponding open face. The stabilizer subgroups $G_\eta\subset G$ coincide for all $\eta\in\sigma$; write $G_\sigma$ for this common subgroup. Define
	\begin{align*}
		\ell^{\Sigma} &\coloneqq \dim\bigl(G_\sigma/[G_\sigma,G_\sigma]\bigr)=\dim\sigma,\\
		u^{\Sigma} &\coloneqq \frac12\bigl(\dim G-\dim G_\sigma\bigr),\\
		\text{and}\quad b^{\Sigma} &\coloneqq \ell^{\Sigma}+u^{\Sigma}
		=\frac12\bigl(\dim G-\dim[G_\sigma,G_\sigma]+\dim\sigma\bigr).
	\end{align*}
	Note that $2u^{\Sigma}$ is the dimension of every coadjoint orbit in $\Sigma$. In the case $\Sigma=\g^*_{\mathrm{reg}}$, consider the quantities $r$, $u$, and $b$ from Section \ref{sec:CW-gelfand-cetlin-datum}. Observe that $\ell^{\Sigma}=r$, $u^{\Sigma}=u$, and $b^{\Sigma}=b$ in that case.
	
	Consider the tori
	$$
	\T_{\mathrm{small}}^{\Sigma}\coloneqq\operatorname{U}(1)^{\ell^{\Sigma}},\qquad
	\T_{\mathrm{int}}^{\Sigma}\coloneqq\operatorname{U}(1)^{u^{\Sigma}},\quad\text{and}\quad
	\T_{\mathrm{big}}^{\Sigma}\coloneqq\T_{\mathrm{small}}^{\Sigma}\times\T_{\mathrm{int}}^{\Sigma}.
	$$
	Using the identification $\R^d\longrightarrow(i\R)^d$, $x\longmapsto ix$, we identify the Lie algebra of $\operatorname{U}(1)^d$ with $\R^d$. We also identify $\mathbb{R}^d$ and $(\mathbb{R}^d)^*$ via the dot product. It follows that
	$$
	\R_{\mathrm{small}}^{\Sigma}\coloneqq\R^{\ell^{\Sigma}},\qquad
	\R_{\mathrm{int}}^{\Sigma}\coloneqq\R^{u^{\Sigma}},\quad\text{and}\quad
	\R_{\mathrm{big}}^{\Sigma}\coloneqq\R_{\mathrm{small}}^{\Sigma}\times\R_{\mathrm{int}}^{\Sigma}=\R^{b^{\Sigma}}
	$$
	are the duals of the Lie algebras of $\T_{\mathrm{small}}^{\Sigma}$, $\T_{\mathrm{int}}^{\Sigma}$, and $\T_{\mathrm{big}}^{\Sigma}$, respectively.
	
	The stratum $\Sigma$ is a Poisson submanifold of $\g^*$. Lemma \ref{strata_cotangent_space} gives a canonical identification
	$$
	T_\xi^*\Sigma\cong\g/[\g_\xi,\g_\xi]
	\quad\text{for all }\xi\in\Sigma.
	$$
	If a $G$-invariant function $\lambda:\g^*\longrightarrow\R$ is smooth on $\Sigma$, its differential annihilates $T_\xi(G\cdot\xi)$. We may therefore regard $(\mathrm{d}(\lambda\vert_\Sigma))_\xi$ as an element of $\g_\xi/[\g_\xi,\g_\xi]$. This is the Lie algebra of the torus $G_\xi/[G_\xi,G_\xi]$. With these considerations in mind, we have the following definition.
	
	\begin{definition}\label{gen_gsdatum}
		A \textit{generalized Gelfand--Cetlin datum} for $\Sigma$ is a pair $(\lambda_{\mathrm{big}}^{\Sigma},\Sigma_{\mathrm{reg}})$ consisting of a continuous map
		$$
		\lambda_{\mathrm{big}}^{\Sigma}\coloneqq(\lambda_1,\dots,\lambda_{b^{\Sigma}}):
		\g^*\longrightarrow\R_{\mathrm{big}}^{\Sigma}
		$$
		and a subset $\Sigma_{\mathrm{reg}}\subset\Sigma$ satisfying the following conditions:
		\begin{enumerate}[(i)]
			\item the functions $\lambda_1,\dots,\lambda_{\ell^{\Sigma}}$ are $G$-invariant on $\g^*$ and smooth on $\Sigma$;
			\item the elements $(\mathrm{d}(\lambda_1\vert_\Sigma))_\xi,\dots,(\mathrm{d}(\lambda_{\ell^{\Sigma}}\vert_\Sigma))_\xi$ form a $\Z$-basis of $\Lambda_{\g_\xi/[\g_\xi,\g_\xi]}$ for all $\xi\in\Sigma_{\mathrm{reg}}$;
			\item the subset $\Sigma_{\mathrm{reg}}$ is open and dense in $\Sigma$;
			\item the restriction $\lambda_{\mathrm{big}}^{\Sigma}\vert_{\Sigma_{\mathrm{reg}}}:\Sigma_{\mathrm{reg}}\longrightarrow\R_{\mathrm{big}}^{\Sigma}$ is a smooth submersion and a moment map for a Poisson Hamiltonian $\T_{\mathrm{big}}^{\Sigma}$-action;
			\item the subset $\Sigma_{\mathrm{reg}}$ is a union of fibers of $\lambda_{\mathrm{big}}^{\Sigma}\vert_\Sigma$, and the map
			$$
			\lambda_{\mathrm{big}}^{\Sigma}\vert_{\Sigma_{\mathrm{reg}}}:
			\Sigma_{\mathrm{reg}}\longrightarrow\lambda_{\mathrm{big}}^{\Sigma}(\Sigma_{\mathrm{reg}})
			$$
			is a principal $\T_{\mathrm{int}}^{\Sigma}$-bundle for the restriction of the action in Condition \textup{(iv)};
			\item for every non-empty, connected Hamiltonian $G$-space $(M,\omega,\mu)$ with principal orbit-type stratum $\Sigma$, the map
			$$
			(\lambda_{\mathrm{big}}^{\Sigma}\circ\mu)\vert_{\mu^{-1}(\Sigma_{\mathrm{reg}})}:
			\mu^{-1}(\Sigma_{\mathrm{reg}})\longrightarrow\R_{\mathrm{big}}^{\Sigma}
			$$
			is a moment map for a Hamiltonian $\T_{\mathrm{big}}^{\Sigma}$-action on the open subset $\mu^{-1}(\Sigma_{\mathrm{reg}})\subset M$.
		\end{enumerate}
	\end{definition}
	
	Once a Hamiltonian $G$-space $(M,\omega,\mu)$ with principal orbit-type stratum $\Sigma$ is specified, we use the notation
	$$
	M_{\mathrm{reg}}\coloneqq\mu^{-1}(\Sigma_{\mathrm{reg}}),
	\quad\text{and}\quad
	\lambda_M\coloneqq(\lambda_{\mathrm{big}}^{\Sigma}\circ\mu)\vert_{M_{\mathrm{reg}}}.
	$$
	
	\begin{remark}
		A Gelfand--Cetlin datum in the sense of Definition \ref{gsdatum} is a generalized Gelfand--Cetlin datum for $\g^*_{\mathrm{reg}}$. To check the saturation clause in Condition (v) of Definition \ref{gen_gsdatum}, write $B\coloneqq\lambda_{\mathrm{big}}(\g^*_{\mathrm{s-reg}})$. The principal bundle over $B$ is proper because its structure group is compact, and $B$ is open by the submersion condition. If a point outside $\g^*_{\mathrm{s-reg}}$ had its image in $B$, density would give a sequence in $\g^*_{\mathrm{s-reg}}$ converging to that point, with images in a compact subset of $B$. Properness would force its limit to lie in $\g^*_{\mathrm{s-reg}}$, yielding a contradiction. The other conditions follow directly from Definition \ref{gsdatum}.
	\end{remark}
	
	\subsection{Existence of generalized Gelfand--Cetlin data}\label{Subsection: Existence}
	We use the toric degeneration of Hoffman--Lane \cite{hoffman2025stratified} to construct a generalized Gelfand--Cetlin datum for a fixed orbit-type stratum $\Sigma=G\cdot\sigma\subset\g^*$.
	
	\begin{proposition}\label{prop:gc-existence}
		Every orbit-type stratum $\Sigma\subset\g^*$ admits a generalized Gelfand--Cetlin datum.
	\end{proposition}
	\begin{proof}
		Equip $T^*G$ with the symplectic form $-\mathrm{d}\theta$, where $\theta$ is the tautological one-form. In the left trivialization $T^*G=G\times\g^*$, left multiplication has moment map $\mu_L(g,\xi)=\Ad_g^*(\xi)$, and inverse right multiplication has moment map $\mu_R(g,\xi)=-\xi$. Following Hoffman--Lane \cite[Section 6.1]{hoffman2025stratified}, we implode the latter action with respect to $-\t_+^*$. We then let each $t\in T$ act by the map that the residual action assigns to $t^{-1}$. The resulting universal symplectic implosion has the decomposition \cite{gui-jef-sja:02}
		\begin{equation}\label{eq:gc-existence-implosion}
			(T^*G)_{\mathrm{impl}}
			=\bigsqcup_{\tau}(T^*G)_{\mathrm{impl},\tau},
			\qquad
			(T^*G)_{\mathrm{impl},\tau}
			\coloneqq \bigl(G/[G_\tau,G_\tau]\bigr)\times\tau,
		\end{equation}
		where $\tau$ ranges over the open faces of $\t_+^*$. On the symplectic piece corresponding to $\sigma$, the resulting $T$-action is given by $t\cdot([g],\xi)=([gt],\xi)$ and has moment map $([g],\xi)\longmapsto\xi$. The surjection $T\longrightarrow G_\sigma/[G_\sigma,G_\sigma]$ has kernel $T\cap[G_\sigma,G_\sigma]$, and the quotient torus acts freely on this piece. The left moment map descends to $(T^*G)_{\mathrm{impl}}$ and restricts on the piece corresponding to $\sigma$ to the Poisson submersion
		\begin{equation}\label{eq:gc-existence-q}
			\mu_L:(T^*G)_{\mathrm{impl},\sigma}\longrightarrow\Sigma,
			\qquad ([g],\xi)\longmapsto\Ad_g^*(\xi).
		\end{equation}
		Under the identification $\Sigma\simeq(G/G_\sigma)\times\sigma$, the map $\mu_L$ sends $([g],\eta)$ to $([g],\eta)$. It follows that $\mu_L$ is a principal $G_\sigma/[G_\sigma,G_\sigma]$-bundle whose fibers are the orbits of right multiplication on the first factor. Pullback by $\mu_L$ identifies smooth functions on $\Sigma$ with smooth functions on $(T^*G)_{\mathrm{impl},\sigma}$ that are invariant under the residual $T$-action. Since $\mu_L$ is a Poisson submersion for the restricted Lie--Poisson bracket on $\Sigma$, this identification carries that bracket to the bracket of invariant functions on $(T^*G)_{\mathrm{impl},\sigma}$.
		
		Let $(T^*G)_{\mathrm{impl}}^{\mathrm{deg}}$ denote the toric degeneration of the universal implosion constructed in \cite[Example 5.6 and Section 6.2]{hoffman2025stratified}. The acting torus $\T_{\mathrm{deg}}$ has Lie algebra $\t_{\mathrm{deg}}$, and we write $\mu_{\mathrm{deg}}$ for its moment map. Restriction along $T\hookrightarrow\T_{\mathrm{deg}}$ gives the residual $T$-action. Let $\pr_T:\t_{\mathrm{deg}}^*\longrightarrow\t^*$ be the dual of the differential of $T\hookrightarrow\T_{\mathrm{deg}}$. The \textit{sweeping map} sends $\xi\in\g^*$ to the unique point of $G\cdot\xi$ in the fundamental Weyl chamber $\t_+^*$. In this proof, write $\pi:\g^*\longrightarrow\t^*$ for the composition of the sweeping map with the inclusion $\t_+^*\hookrightarrow\t^*$. Hoffman--Lane construct continuous maps
		\[
		\varphi_{\mathrm{deg}}:(T^*G)_{\mathrm{impl}}\longrightarrow (T^*G)_{\mathrm{impl}}^{\mathrm{deg}}
		\qquad\text{and}\qquad
		\lambda_{\mathrm{deg}}:\g^*\longrightarrow\t_{\mathrm{deg}}^*
		\]
		satisfying
		\begin{equation}\label{eq:gc-existence-universal}
			\mu_{\mathrm{deg}}\circ\varphi_{\mathrm{deg}}=\lambda_{\mathrm{deg}}\circ\mu_L
			\quad\text{and}\quad
			\pr_T\circ\lambda_{\mathrm{deg}}=\pi.
		\end{equation}
		They also show that $\varphi_{\mathrm{deg}}$ is surjective, and that $\mu_{\mathrm{deg}}$ has connected fibers \cite[Theorem 6.2 and Proposition 6.4]{hoffman2025stratified}.
		
		Write $(T^*G)_{\mathrm{impl},\sigma}^{\mathrm{deg}}$ for the inverse image of $\sigma$ under $\pr_T\circ\mu_{\mathrm{deg}}$. For a compact connected semisimple group, the smooth open subset used in \cite[Proposition 6.5]{hoffman2025stratified} is the smooth locus of this toric piece. For our group $G$, one sets $Z=Z(G)^\circ$, takes the product of the corresponding smooth locus for $[G,G]$ with $T^*Z$, and divides by the finite kernel of the multiplication map $[G,G]\times Z\longrightarrow G$. The resulting open subset is denoted by $(T^*G)_{\mathrm{impl},\sigma,\mathrm{sm}}^{\mathrm{deg}}$. By the same proposition, its preimage under $\varphi_{\mathrm{deg}}$ is a connected open dense subset of $(T^*G)_{\mathrm{impl},\sigma}$ on which $\varphi_{\mathrm{deg}}$ restricts to a symplectomorphism intertwining the residual $T$-actions.
		
		Quotienting $\T_{\mathrm{deg}}$ by the kernel of its action on $(T^*G)_{\mathrm{impl},\sigma,\mathrm{sm}}^{\mathrm{deg}}$ gives a torus $\T_{\mathrm{deg}}^{\Sigma}$ acting effectively on this piece. The action has complexity zero \cite[Theorem 5.5 and the proof of Proposition 6.11]{hoffman2025stratified}, so
		\[
		\dim\T_{\mathrm{deg}}^{\Sigma}
		=\frac12\dim(T^*G)_{\mathrm{impl},\sigma}
		=\frac12\bigl(\dim G-\dim[G_\sigma,G_\sigma]+\dim\sigma\bigr)
		=b^{\Sigma}.
		\]
		On $(T^*G)_{\mathrm{impl},\sigma,\mathrm{sm}}^{\mathrm{deg}}$, the residual $T$-action is given by the composite
		\[
		T\hookrightarrow\T_{\mathrm{deg}}\longrightarrow\T_{\mathrm{deg}}^{\Sigma}.
		\]
		An element $t\in T$ maps to the identity in $\T_{\mathrm{deg}}^{\Sigma}$ precisely when it fixes every point of the smooth toric piece. Under the $T$-equivariant symplectomorphism above, $t$ sends $([g],\xi)$ to $([gt],\xi)$, and hence fixes $([g],\xi)$ precisely when the cosets $gt[G_\sigma,G_\sigma]$ and $g[G_\sigma,G_\sigma]$ coincide. Multiplying on the left by $g^{-1}$ gives the condition $t\in[G_\sigma,G_\sigma]$, which is independent of $g$ and $\xi$. The composite $T\hookrightarrow\T_{\mathrm{deg}}\longrightarrow
		\T_{\mathrm{deg}}^\Sigma$ therefore has kernel $T\cap[G_\sigma,G_\sigma]$. Its image is canonically isomorphic to $T/(T\cap[G_\sigma,G_\sigma])\simeq G_\sigma/[G_\sigma,G_\sigma]$, the structure torus of the bundle in Equation \ref{eq:gc-existence-q}. Choose a $\Z$-basis $\phi_1,\ldots,\phi_{\ell^{\Sigma}}$ of the integral lattice of this subtorus. Since the lattice of a subtorus is primitive, one may extend the chosen basis to a $\Z$-basis $\phi_1,\ldots,\phi_{b^{\Sigma}}$ of the integral lattice of $\T_{\mathrm{deg}}^{\Sigma}$. This basis identifies the effective torus with $\T_{\mathrm{big}}^{\Sigma}$ and $G_\sigma/[G_\sigma,G_\sigma]$ with the first factor $\T_{\mathrm{small}}^{\Sigma}$.
		
		Differentiating the quotient homomorphism $\T_{\mathrm{deg}}\longrightarrow\T_{\mathrm{deg}}^\Sigma$ at the identity gives a surjection $\t_{\mathrm{deg}}\longrightarrow\t_{\mathrm{deg}}^\Sigma$. Choose a lift $\widetilde\phi_j\in\t_{\mathrm{deg}}$ of each $\phi_j$ under this surjection as follows. For $j\leq\ell^\Sigma$, the identification of the image of $T$ with
		$G_\sigma/[G_\sigma,G_\sigma]$ lets us view $\phi_j$ as a class in $\g_\sigma/[\g_\sigma,\g_\sigma]$. The decomposition
		\[
		\g_\sigma=\z(\g_\sigma)\oplus[\g_\sigma,\g_\sigma]
		\]
		gives this class a unique representative in $\z(\g_\sigma)\subset\t$. We take $\widetilde\phi_j$ to be the image of this representative under $\t\hookrightarrow\t_{\mathrm{deg}}$. For $j>\ell^\Sigma$,
		we choose any preimage of $\phi_j$. These lifts define the maps
		\begin{align*}
			\pr_{\mathrm{big}}^{\Sigma}:\t_{\mathrm{deg}}^*&\longrightarrow\R_{\mathrm{big}}^{\Sigma},
			& \eta&\longmapsto
			\bigl(\eta(\widetilde\phi_1),\ldots,
			\eta(\widetilde\phi_{b^{\Sigma}})\bigr),\\
			\lambda_{\mathrm{big}}^{\Sigma}&\coloneqq \pr_{\mathrm{big}}^{\Sigma}\circ\lambda_{\mathrm{deg}},
			& \lambda_j(\xi)&=\lambda_{\mathrm{deg}}(\xi)(\widetilde\phi_j).
		\end{align*}
		The component functions are continuous on $\g^*$. By Equation \ref{eq:gc-existence-universal}, the first $\ell^{\Sigma}$ are obtained by pairing $\pi$ with the chosen central representatives, and are therefore $G$-invariant and smooth on $\Sigma$. At $\xi\in\sigma$, the differentials of their restrictions to $\Sigma$ vanish on $T_\xi(G\cdot\xi)$ and agree with the central representatives on $T_\xi\sigma$ because $\pi\vert_\sigma$ is the identity. They therefore represent $\phi_1,\ldots,\phi_{\ell^{\Sigma}}$ in $\g_\sigma/[\g_\sigma,\g_\sigma]$, and conjugation gives the required integral basis at every point of $\Sigma$. This proves Conditions (i) and (ii) of Definition \ref{gen_gsdatum}.
		
		Denote the free locus of $\T_{\mathrm{big}}^{\Sigma}$ on $(T^*G)_{\mathrm{impl},\sigma,\mathrm{sm}}^{\mathrm{deg}}$ by $(T^*G)_{\mathrm{impl},\sigma,\mathrm{free}}^{\mathrm{deg}}$. Since $\T_{\mathrm{big}}^\Sigma$ is abelian, all points in the open dense principal stratum have the same stabilizer. This subgroup fixes the entire smooth toric piece by continuity, and is therefore trivial by effectiveness. It follows that the free locus is the open dense locus of principal orbit-type. For $x\in(T^*G)_{\mathrm{impl},\sigma,\mathrm{sm}}^{\mathrm{deg}}$,
		write $(\mathrm{d}\mu_{\mathrm{deg}})_x$ for the differential at $x$
		of $\mu_{\mathrm{deg}}\vert_{(T^*G)_{\mathrm{impl},\sigma,\mathrm{sm}}^{\mathrm{deg}}}$. By the moment map identity, its kernel is the symplectic orthogonal of $T_x(\T_{\mathrm{deg}}\cdot x)$. At a free point, the orbit $\T_{\mathrm{deg}}\cdot x=\T_{\mathrm{big}}^{\Sigma}\cdot x$ has dimension $b^{\Sigma}$, so the symplectic orthogonal of its tangent space has codimension $b^{\Sigma}$. It follows that $(\mathrm{d}\mu_{\mathrm{deg}})_x$ has rank $b^{\Sigma}$ throughout the free locus. By the constant rank theorem, the fiber is a manifold of dimension
		$2b^{\Sigma}-b^{\Sigma}=b^{\Sigma}$ near each such $x$. The torus orbit has the same dimension as the fiber restricted to the free locus in the smooth toric piece. The dimension argument applies at every point of the orbit, so the orbit is open in that restricted fiber. The full fiber lies in $(T^*G)_{\mathrm{impl},\sigma}^{\mathrm{deg}}$ because $\pr_T\circ\mu_{\mathrm{deg}}$ is constant on it with value in $\sigma$. The free locus is open in the smooth toric piece, which is itself open in $(T^*G)_{\mathrm{impl},\sigma}^{\mathrm{deg}}$. Thus the orbit is open in the full fiber as well. The orbit is compact and hence closed in the full fiber, which is connected. It follows that every fiber meeting the free locus is a single free orbit.
		
		Define
		\begin{equation}\label{eq:gc-existence-regular}
			\Sigma_{\mathrm{reg}}\coloneqq
			\lambda_{\mathrm{deg}}^{-1}\bigl(\mu_{\mathrm{deg}}((T^*G)_{\mathrm{impl},\sigma,\mathrm{free}}^{\mathrm{deg}})\bigr).
		\end{equation}
		Given $\xi\in\Sigma_{\mathrm{reg}}$, one may choose $x\in(T^*G)_{\mathrm{impl},\sigma,\mathrm{free}}^{\mathrm{deg}}$ such that $\lambda_{\mathrm{deg}}(\xi)=\mu_{\mathrm{deg}}(x)$. Equation \ref{eq:gc-existence-universal} and the definition of the piece corresponding to $\sigma$ then give
		\[
		\pi(\xi)=\pr_T\bigl(\lambda_{\mathrm{deg}}(\xi)\bigr)
		=\pr_T\bigl(\mu_{\mathrm{deg}}(x)\bigr)\in\sigma.
		\]
		The point $\pi(\xi)$ is the unique point of $G\cdot\xi$ in the fundamental Weyl chamber, so $\xi\in G\cdot\sigma=\Sigma$. Applying the first identity in Equation \ref{eq:gc-existence-universal} and the description of the fibers of $\mu_{\mathrm{deg}}$, one obtains
		\[
		\mu_L^{-1}(\Sigma_{\mathrm{reg}})
		=\varphi_{\mathrm{deg}}^{-1}\bigl((T^*G)_{\mathrm{impl},\sigma,\mathrm{free}}^{\mathrm{deg}}\bigr).
		\]
		The right-hand side is open and dense in $(T^*G)_{\mathrm{impl},\sigma}$. Its image under the open surjection $\mu_L$ is therefore open and dense in $\Sigma$. This proves Condition (iii).
		
		The restriction of $\varphi_{\mathrm{deg}}$ to $\mu_L^{-1}(\Sigma_{\mathrm{reg}})$ is a $\T_{\mathrm{small}}^{\Sigma}$-equivariant symplectomorphism onto the free toric locus. For a point $x$ in the free toric locus, write $([g],\eta)$ for its inverse image under this symplectomorphism and $[x]$ for its $\T_{\mathrm{small}}^{\Sigma}$-orbit. The inverse symplectomorphism is equivariant, and $\mu_L$ is constant on the $\T_{\mathrm{small}}^{\Sigma}$-orbits in $\mu_L^{-1}(\Sigma_{\mathrm{reg}})$. Thus the value $\mu_L([g],\eta)$ depends only on $[x]$. By Equation \ref{eq:gc-existence-q}, this gives the Poisson diffeomorphism
		\[
		\begin{aligned}
			(T^*G)_{\mathrm{impl},\sigma,\mathrm{free}}^{\mathrm{deg}}/\T_{\mathrm{small}}^{\Sigma}
			&\longrightarrow\Sigma_{\mathrm{reg}},\\
			[x]&\longmapsto\mu_L([g],\eta)=\Ad_g^*(\eta).
		\end{aligned}
		\]
		
		We claim that $\lambda_{\mathrm{big}}^{\Sigma}\vert_\Sigma$ and $\lambda_{\mathrm{deg}}\vert_\Sigma$ have the same fibers. An element $v$ in the Lie algebra of the kernel of $\T_{\mathrm{deg}}\longrightarrow\T_{\mathrm{deg}}^{\Sigma}$ acts infinitesimally by zero on the smooth toric piece, so the moment component $\mu_{\mathrm{deg}}^v$ is constant on this connected piece. Equation \ref{eq:gc-existence-universal} gives the same constant for $\lambda_{\mathrm{deg}}^v\circ\mu_L$ on the dense preimage of the piece in $(T^*G)_{\mathrm{impl},\sigma}$. By continuity, this equality holds on all of $(T^*G)_{\mathrm{impl},\sigma}$, and the surjectivity of $\mu_L$ onto $\Sigma$ implies that all elements of $\lambda_{\mathrm{deg}}(\Sigma)\subset\t_{\mathrm{deg}}^*$ take the same value on $v$.
		
		For $\xi\in\Sigma$, the components of $\lambda_{\mathrm{big}}^{\Sigma}(\xi)$ are the values of $\lambda_{\mathrm{deg}}(\xi)$ on the lifts $\widetilde\phi_1,\ldots,\widetilde\phi_{b^{\Sigma}}$. These lifts project to a basis of the quotient Lie algebra $\t_{\mathrm{deg}}^\Sigma$. Together with the kernel of the quotient map $\t_{\mathrm{deg}}\longrightarrow\t_{\mathrm{deg}}^\Sigma$, they span $\t_{\mathrm{deg}}$. Two elements of $\lambda_{\mathrm{deg}}(\Sigma)$ that agree on the lifts therefore agree on all of $\t_{\mathrm{deg}}$, which proves the claim. It follows from Equation \ref{eq:gc-existence-regular} that $\Sigma_{\mathrm{reg}}$ is a union of fibers of $\lambda_{\mathrm{big}}^{\Sigma}\vert_\Sigma$.
		
		On $(T^*G)_{\mathrm{impl},\sigma,\mathrm{free}}^{\mathrm{deg}}$, the moment map $\pr_{\mathrm{big}}^{\Sigma}\circ\mu_{\mathrm{deg}}$ is a submersion whose fibers are the free $\T_{\mathrm{big}}^{\Sigma}$-orbits. Taking the quotient of this principal bundle by the first factor $\T_{\mathrm{small}}^{\Sigma}$ gives a principal $\T_{\mathrm{int}}^{\Sigma}$-bundle with projection
		\[
		\lambda_{\mathrm{big}}^{\Sigma}\vert_{\Sigma_{\mathrm{reg}}}:
		\Sigma_{\mathrm{reg}}\longrightarrow\lambda_{\mathrm{big}}^{\Sigma}(\Sigma_{\mathrm{reg}}).
		\]
		By Equation \ref{eq:gc-existence-universal}, the moment components descend to the functions $\lambda_j\vert_{\Sigma_{\mathrm{reg}}}$. These functions generate the induced action of $\T_{\mathrm{big}}^{\Sigma}$, with its first factor acting trivially. This proves Conditions (iv) and (v).
		
		Finally, consider a non-empty, connected Hamiltonian $G$-space $(M,\omega,\mu)$ with principal orbit-type stratum $\Sigma\subset\g^*$. The toric contraction of $M$ is the reduction of $M_{\mathrm{impl}}\times(T^*G)_{\mathrm{impl}}^{\mathrm{deg}}$ at zero by the fixed maximal torus $T\subset G$ \cite[Definition 6.6]{hoffman2025stratified}. In terms of the residual $T$-actions on the two factors, an element of $T$ acts on the product by
		\[
		t\cdot(m,x)=(t\cdot m,t^{-1}\cdot x).
		\]
		The moment map for this action is $(m,x)\longmapsto\mu_{\mathrm{impl}}(m)-\pr_T(\mu_{\mathrm{deg}}(x))$, where $\mu_{\mathrm{impl}}:M_{\mathrm{impl}}\longrightarrow\t^*$ is the residual moment map on the first factor. By \cite[Propositions 6.8 and 6.9]{hoffman2025stratified}, the contraction map is a symplectomorphism over the piece obtained by reducing $\mu^{-1}(\sigma)\times(T^*G)_{\mathrm{impl},\sigma,\mathrm{sm}}^{\mathrm{deg}}$, and the map induced by $\mu_{\mathrm{deg}}$ on the contraction pulls back to $\lambda_{\mathrm{deg}}\circ\mu$. Every fiber of $\mu_{\mathrm{deg}}$ meeting the free locus lies entirely in it, so the symplectomorphism restricts by Equation \ref{eq:gc-existence-regular} to
		\[
		\mu^{-1}(\Sigma_{\mathrm{reg}})
		\simeq \bigl(\mu^{-1}(\sigma)\times(T^*G)_{\mathrm{impl},\sigma,\mathrm{free}}^{\mathrm{deg}}\bigr)\sll{0}T .
		\]
		On the second factor, the residual $T$-action factors through $\T_{\mathrm{small}}^{\Sigma}\subset\T_{\mathrm{big}}^{\Sigma}$. The action of $\T_{\mathrm{big}}^{\Sigma}$ on this factor commutes with the displayed $T$-action on the product because $\T_{\mathrm{big}}^{\Sigma}$ is abelian. It also preserves $\mu_{\mathrm{deg}}$ and hence the zero level set defined by $\mu_{\mathrm{impl}}(m)=\pr_T(\mu_{\mathrm{deg}}(x))$. On the reduction $\bigl(\mu^{-1}(\sigma)\times(T^*G)_{\mathrm{impl},\sigma,\mathrm{free}}^{\mathrm{deg}}\bigr)\sll{0}T$, the induced action is $s\cdot[(m,x)]=[(m,s\cdot x)]$ for $s\in\T_{\mathrm{big}}^{\Sigma}$, where $[(m,x)]$ denotes the $T$-orbit of $(m,x)$ in that zero level set. Pulling this action back along the displayed symplectomorphism gives a Hamiltonian $\T_{\mathrm{big}}^{\Sigma}$-action on $\mu^{-1}(\Sigma_{\mathrm{reg}})$ with moment map $(\lambda_{\mathrm{big}}^{\Sigma}\circ\mu)\vert_{\mu^{-1}(\Sigma_{\mathrm{reg}})}$. This proves Condition (vi).
	\end{proof}
	
	\subsection{Invariant components from the sweeping map} \label{construction_of_GS_datum}
	We aim to provide a method for constructing generalized Gelfand--Cetlin data. 
	Our first step is to construct functions $\lambda_1, \dots, \lambda_{\ell^{\Sigma}}:\g^* \longrightarrow \R$ that satisfy Conditions (i) and (ii) of Definition \ref{gen_gsdatum}. 
	To this end, consider a compact connected Lie group $G$ with Lie algebra $\g$ and an orbit-type stratum $\Sigma\subset\g^*$.
	Choose a $G$-invariant inner product on $\g$, a Cartan subalgebra $\t \subset \g$, and a fundamental Weyl chamber $\t_+ \subset \t$. 
	Recall that the inner product allows us to identify $\t_+$ with a fundamental domain $\t_+^* \subset \g^*$ for the $G$-action on $\g^*$.
	We write $\pi:\g^*\longrightarrow\t_+^*$ for the sweeping map, i.e. the continuous surjection satisfying
	$$
	(G\cdot \xi) \cap \t_+^* = \{\pi(\xi)\} \quad \text{for all } \xi \in \g^*.
	$$
	The fibers of $\pi$ are coadjoint orbits of $G$, and $\pi(\Sigma) = \sigma$, where $\sigma$ is the open face of $\t_+^*$ corresponding to the orbit-type stratum $\Sigma$. 
	This gives the following Cartesian square:
	\begin{center}
		\begin{tikzcd}
			\Sigma \arrow[r, hookrightarrow]\arrow[d, swap, "\pi\big\vert_{\Sigma}"]& \g^*  \arrow[d, "\pi"] \\
			\sigma \arrow[r, hookrightarrow]& \t_+^*
		\end{tikzcd}.
	\end{center}
	Note that the left vertical map $\pi\big\vert_\Sigma: \Sigma \longrightarrow \sigma\subset \t_+^*$ is a smooth surjective submersion. 
	
	Recall that $\g_\sigma \subset \g$ denotes the common $\g$-centralizer for elements of $\sigma$.
	Choose a $\Z$-basis $\{\phi_1, \dots, \phi_{\ell^{\Sigma}} \}$ of the lattice $\Lambda_{\g_\sigma/ [\g_\sigma, \g_\sigma]}$. 
	At the same time, recall the canonical Lie algebra isomorphism $\g_\sigma\big/ [\g_\sigma, \g_\sigma]\cong \mathfrak{z}(\g_\sigma)$. 
	This allows us to regard $\phi_1, \dots, \phi_{\ell^{\Sigma}}$ as elements of $\z(\g_\sigma)$.
	The pairing between $\g$ and $\g^*$ now allows one to regard $\phi_1, \dots, \phi_{\ell^{\Sigma}}$ as functions on $\t_+^*$.
	For each $k \in \{1, \dots, \ell^{\Sigma}\}$, we define a continuous function
	$$
	\lambda_k\coloneqq\phi_k\circ\pi:\g^* \longrightarrow \R.
	$$
	These functions are $G$-invariant because the sweeping map is so.
	Since the restriction of $\pi$ to $\Sigma$ is smooth, $\lambda_1, \dots, \lambda_{\ell^{\Sigma}}$ are smooth on $\Sigma$ and satisfy Condition (i).
	
	We next verify Condition (ii) of Definition \ref{gen_gsdatum}. The differential of each $\lambda_k\vert_\Sigma$ is regarded as an element of $\g_\xi/[\g_\xi,\g_\xi]$ using the identification in Section \ref{sec:gen-GC-data-definition}. Equivalently, it has a unique representative in $\z(\g_\xi)$.
	
	\begin{proposition}\label{Proposition: Annoying}
		If $\xi \in \Sigma$, then $\{(\mathrm{d}(\lambda_1\big\vert_\Sigma))_{\xi}, \dots, (\mathrm{d} (\lambda_{\ell^{\Sigma}}\big\vert_\Sigma))_{\xi}\}$ is a $\Z$-basis of the lattice $\Lambda_{\g_\xi/[\g_\xi, \g_\xi]}$.
	\end{proposition}
	\begin{proof}
		Choose $g\in G$ such that $\Ad_g^*(\xi)\in\sigma$. The $G$-invariance of $\lambda_k$ gives
		$$
		(\mathrm{d}(\lambda_k\vert_\Sigma))_\xi
		=\Ad_{g^{-1}}\bigl((\mathrm{d}(\lambda_k\vert_\Sigma))_{\Ad_g^*(\xi)}\bigr),
		$$
		where the adjoint map is understood on the corresponding quotient Lie algebras. Conjugation induces an isomorphism of the stabilizer quotient tori and hence of their integral lattices. We may therefore assume that $\xi\in\sigma$.
		It follows that
		$$
		T_\xi\Sigma=T_\xi(G\cdot\xi)\oplus T_\xi\sigma.
		$$
		By $G$-invariance, $(\mathrm{d}(\lambda_k\vert_\Sigma))_\xi$ vanishes on $T_\xi(G\cdot\xi)=\operatorname{ann}_{\g^*}(\g_\sigma)$. The element $\phi_k\in\z(\g_\sigma)$ annihilates the same summand. The differential and the linear form induced by $\phi_k$ agree on $T_\xi\sigma$ because $\lambda_k=\phi_k\circ\pi$ and $\pi\vert_\sigma=\operatorname{id}_\sigma$. Under the identification of $T_\xi^*\Sigma$ with $\g/[\g_\sigma,\g_\sigma]$ implied by Lemma \ref{strata_cotangent_space}, the differential $(\mathrm{d}(\lambda_k\vert_\Sigma))_\xi$ is the class of $\phi_k$ in $\g_\sigma/[\g_\sigma,\g_\sigma]$. The chosen elements $\phi_1,\dots,\phi_{\ell^{\Sigma}}$ form the required integral basis.
	\end{proof}
	
	\subsection{Commuting Hamiltonians from Thimm's method}\label{Subsection: Construction}
	We now use Thimm's method to construct further commuting functions, following \cite{guillemin1983collective,lane2018Convexity}. We retain the stratum $\Sigma=G\cdot\sigma$ and the functions $\lambda_1,\dots,\lambda_{\ell^{\Sigma}}$ from Section \ref{construction_of_GS_datum}. The choices below depend on $\Sigma$ and a chain of connected closed subgroups of $G$, independently of the Hamiltonian $G$-space to which the resulting data will be applied.
	
	Let $G=G_0\supset G_1\supset\cdots\supset G_m$ be a descending chain of connected closed subgroups, with Lie algebras $\g=\g_0\supset\g_1\supset\cdots\supset\g_m$. For each $k$, choose a Cartan subalgebra $\t_k\subset\g_k$ and a fundamental Weyl chamber $\t_{k,+}^*\subset\g_k^*$, taking $\t_0=\t$ and $\t_{0,+}^*=\t_+^*$. Write $\pi_k:\g_k^*\longrightarrow\t_{k,+}^*$ for the corresponding sweeping map and $p_k:\g^*\longrightarrow\g_k^*$ for the dual of the inclusion. It follows that $\pi_0=\pi$ and $p_0=\operatorname{id}_{\g^*}$.
	
	Consider the Hamiltonian $G$-space $(T^*G)_{\mathrm{impl},\sigma}$ and its surjective moment map $\mu_L:(T^*G)_{\mathrm{impl},\sigma}\longrightarrow\Sigma$ from Section \ref{Subsection: Existence}. For the restricted $G_k$-action, the moment map is $p_k\circ \mu_L$. Let $\Sigma_k\subset\g_k^*$ be the principal orbit-type stratum for the restricted $G_k$-action on $(T^*G)_{\mathrm{impl},\sigma}$, and let $\sigma_k\coloneqq\Sigma_k\cap\t_{k,+}^*$ be the corresponding open face. Set $\ell^{\Sigma_k}\coloneqq\dim\sigma_k$, so that $\Sigma_0=\Sigma$ and $\ell^{\Sigma_0}=\ell^{\Sigma}$. For $1\leq k\leq m$, choose a $\Z$-basis
	$$
	\{\phi_{k,1},\dots,\phi_{k,\ell^{\Sigma_k}}\}
	\subset
	\Lambda_{(\g_k)_{\sigma_k}/[(\g_k)_{\sigma_k},(\g_k)_{\sigma_k}]},
	$$
	and represent these elements in $\z((\g_k)_{\sigma_k})$ as in Section \ref{construction_of_GS_datum}. The functions
	$$
	\nu_{k,j}\coloneqq\phi_{k,j}\circ\pi_k:\g_k^*\longrightarrow\R
	\quad\text{and}\quad
	\lambda^{\mathrm{Th}}_{k,j}\coloneqq\nu_{k,j}\circ p_k:\g^*\longrightarrow\R
	$$
	are continuous, and $\nu_{k,j}$ is smooth on $\Sigma_k$.
	
	Set $c\coloneqq\ell^{\Sigma_0}+\cdots+\ell^{\Sigma_m}$, and arrange these functions as
	$$
	\lambda^{\mathrm{Th}}
	\coloneqq
	(\lambda_1,\dots,\lambda_{\ell^{\Sigma}},
	\lambda^{\mathrm{Th}}_{1,1},\dots,\lambda^{\mathrm{Th}}_{1,\ell^{\Sigma_1}},\dots,
	\lambda^{\mathrm{Th}}_{m,1},\dots,\lambda^{\mathrm{Th}}_{m,\ell^{\Sigma_m}}):
	\g^*\longrightarrow\R^c.
	$$
	This map is smooth on
	$$
	\Sigma^{\mathrm{Th}}\coloneqq\Sigma\cap\bigcap_{k=1}^m p_k^{-1}(\Sigma_k).
	$$
	For each $k$, the set $(p_k\circ\mu_L)^{-1}(\Sigma_k)$ is open and dense in $(T^*G)_{\mathrm{impl},\sigma}$ by Theorem \ref{sympl_cross_section}. The intersection of these sets over $1\leq k\leq m$ is $\mu_L^{-1}(\Sigma^{\mathrm{Th}})$. Since $\mu_L$ is an open surjection onto $\Sigma$, the subset $\Sigma^{\mathrm{Th}}$ is open and dense in $\Sigma$. Define
	$$
	\Sigma^{\mathrm{Th}}_{\mathrm{reg}}\coloneqq
	\{\xi\in\Sigma^{\mathrm{Th}}\mid
	(\mathrm{d}(\lambda^{\mathrm{Th}}\vert_{\Sigma^{\mathrm{Th}}}))_\xi
	\text{ is surjective}\}.
	$$
	
	\begin{proposition}\label{prop:thimm-stratum}
		The restriction of $\lambda^{\mathrm{Th}}$ to $\Sigma^{\mathrm{Th}}$ is a moment map for a Poisson Hamiltonian $\operatorname{U}(1)^c$-action. For every non-empty, connected Hamiltonian $G$-space $(M,\omega,\mu)$ with principal orbit-type stratum $\Sigma$, its pullback is a moment map for a Hamiltonian $\operatorname{U}(1)^c$-action on $\mu^{-1}(\Sigma^{\mathrm{Th}})$. The $\operatorname{U}(1)^c$-actions on $\Sigma^{\mathrm{Th}}$ and $\mu^{-1}(\Sigma^{\mathrm{Th}})$ preserve $\Sigma^{\mathrm{Th}}_{\mathrm{reg}}$ and $\mu^{-1}(\Sigma^{\mathrm{Th}}_{\mathrm{reg}})$, respectively.
	\end{proposition}
	\begin{proof}
		We construct the torus actions by proving completeness and $2\pi$-periodicity for the flows pulled back along $\mu_L$ and $\mu$, and then descending the former to $\Sigma^{\mathrm{Th}}$. A standard invocation of Thimm's method reveals that the component functions Poisson-commute \cite{guillemin1983collective,lane2018Convexity}.
		
		For $\xi\in\Sigma^{\mathrm{Th}}$ and $1\leq k\leq m$, we identify a subgroup of $G_k$ fixing $\mu_L^{-1}(\xi)$ pointwise. Fix $x\in\mu_L^{-1}(\xi)$. The stabilizer of $x$ for the $G$-action is $[G_\xi,G_\xi]$. Since $p_k(\xi)$ belongs to the principal orbit-type stratum $\Sigma_k\subset\g_k^*$ for the Hamiltonian $G_k$-space $(T^*G)_{\mathrm{impl},\sigma}$, Proposition \ref{trivial_action} gives
		\begin{equation}\label{eq:thimm-stabilizer-inclusion}
			[(G_k)_{p_k(\xi)},(G_k)_{p_k(\xi)}]
			\subset [G_\xi,G_\xi].
		\end{equation}
		
		We now check that the Hamiltonian flows of $\lambda^{\mathrm{Th}}_{k,j}\vert_{\Sigma^{\mathrm{Th}}}$ and $\lambda^{\mathrm{Th}}_{k,j}\circ\mu$ are complete on $\Sigma^{\mathrm{Th}}$ and $\mu^{-1}(\Sigma^{\mathrm{Th}})$, respectively. Identify the abelianization of $(\g_k)_{p_k(\xi)}$ with its center, and let $v$ be the image of $(\mathrm d(\nu_{k,j}\vert_{\Sigma_k}))_{p_k(\xi)}$ under this identification. For $p\in\mu^{-1}(\xi)$, the map $p_k\circ\mu$ takes the open neighborhood $\mu^{-1}(\Sigma^{\mathrm{Th}})$ of $p$ into $\Sigma_k$. The chain rule gives
		\[
		\mathrm d(\lambda_{k,j}^{\mathrm{Th}}\circ\mu)_p
		=\mathrm d\langle p_k\circ\mu,v\rangle_p.
		\]
		The moment map identity therefore identifies the Hamiltonian vector at $p$ with $v_M(p)$. The same calculation applies to $\mu_L$. Write $\xi(t)$ for the Hamiltonian trajectory of $\lambda^{\mathrm{Th}}_{k,j}\vert_{\Sigma^{\mathrm{Th}}}$ starting at $\xi$. It has the form $\xi(t)=\Ad_{g(t)}^*(\xi)$ with $g(t)=\exp(-tv)$. In particular, $g(t)$ lies in $(G_k)_{p_k(\xi)}$. Consider the curves $t\longmapsto p_i(\xi(t))$ in $\g_i^*$. The subgroup $(G_k)_{p_k(\xi)}$ fixes $p_k(\xi)$, so $p_k(\xi(t))=p_k(\xi)$. For $i\geq k$, the factorization of $p_i$ through $p_k$ gives $p_i(\xi(t))=p_i(\xi)$. For $i<k$, we have $g(t)\in G_i$. The $G_i$-equivariance of $p_i$ gives $p_i(\xi(t))=\Ad_{g(t)}^*(p_i(\xi))$. It follows that the Hamiltonian trajectory $t\longmapsto\xi(t)$ remains in $\Sigma^{\mathrm{Th}}$ for all time. For $m\in\mu^{-1}(\xi)$, the Hamiltonian trajectory of $\lambda^{\mathrm{Th}}_{k,j}\circ\mu$ starting at $m$ is $t\longmapsto g(t)\cdot m$. The $G$-equivariance of $\mu$ gives $\mu(g(t)\cdot m)=\xi(t)$. This implies that the Hamiltonian trajectory $t\longmapsto g(t)\cdot m$ remains in $\mu^{-1}(\Sigma^{\mathrm{Th}})$.
		
		We next check $2\pi$-periodicity of the flows of $\lambda^{\mathrm{Th}}_{k,j}\circ\mu_L$ and $\lambda^{\mathrm{Th}}_{k,j}\circ\mu$. By Proposition \ref{Proposition: Annoying} applied to $G_k$, the differential $(\mathrm d(\nu_{k,j}\vert_{\Sigma_k}))_{p_k(\xi)}$ belongs to the integral lattice in the abelianization of $(\g_k)_{p_k(\xi)}$, with the normalization of Section \ref{sec:Unitary_and_tori}. Its $2\pi$-multiple therefore exponentiates to the identity in the quotient torus $(G_k)_{p_k(\xi)}/[(G_k)_{p_k(\xi)},(G_k)_{p_k(\xi)}]$. The time-$2\pi$ flow of $\lambda^{\mathrm{Th}}_{k,j}\circ\mu_L$ at $x$ is therefore given by an element of $[(G_k)_{p_k(\xi)},(G_k)_{p_k(\xi)}]$, which fixes $x$ by Equation \ref{eq:thimm-stabilizer-inclusion}. This inclusion also places the subgroup in $[G_\xi,G_\xi]$. By Proposition \ref{trivial_action}, $[G_\xi,G_\xi]$ acts trivially on $\mu^{-1}(\xi)$, so the flow of $\lambda^{\mathrm{Th}}_{k,j}\circ\mu$ returns each point of this fiber to itself at time $2\pi$.
		
		We now check the same periodicity for the pullbacks of the first $\ell^{\Sigma}$ components $\lambda_1,\dots,\lambda_{\ell^{\Sigma}}$. Their flows on $\Sigma^{\mathrm{Th}}$ are trivial because these functions are $G$-invariant. We take the differentials at $\xi$ of the restrictions of these functions to $\Sigma$. Section \ref{sec:gen-GC-data-definition} identifies these differentials with elements of the abelianization of $\g_\xi$. By Proposition \ref{Proposition: Annoying} and the normalization in Section \ref{sec:Unitary_and_tori}, the $2\pi$-multiple of each of these elements exponentiates to the identity in $G_\xi/[G_\xi,G_\xi]$. It follows that the flow of $\lambda_j\circ\mu$ sends every $m\in\mu^{-1}(\xi)$ to $h\cdot m$ at time $2\pi$ for some $h\in[G_\xi,G_\xi]$. Proposition \ref{trivial_action} gives $h\cdot m=m$. The same argument for $\lambda_j\circ\mu_L$ shows that its flow returns every point of $\mu_L^{-1}(\xi)$ to itself at time $2\pi$.
		
		We descend the periodic flows along $\mu_L$ to obtain the torus action on $\Sigma^{\mathrm{Th}}$. On $\mu_L^{-1}(\Sigma^{\mathrm{Th}})$, the group $G_\sigma/[G_\sigma,G_\sigma]$ acts by right multiplication on the first factor: $[h]\cdot([g],\eta)=([gh],\eta)$. This action is symplectic, and its orbits are the fibers of $\mu_L$. The component Hamiltonians are pulled back along $\mu_L$, so their flows commute with this action and descend through $\mu_L$ to $\Sigma^{\mathrm{Th}}$. Together with the flows on $\mu^{-1}(\Sigma^{\mathrm{Th}})$, they give a Poisson Hamiltonian $\operatorname{U}(1)^c$-action on $\Sigma^{\mathrm{Th}}$ and a Hamiltonian $\operatorname{U}(1)^c$-action on $\mu^{-1}(\Sigma^{\mathrm{Th}})$.
		
		It remains to show that the torus actions preserve $\Sigma^{\mathrm{Th}}_{\mathrm{reg}}$ and $\mu^{-1}(\Sigma^{\mathrm{Th}}_{\mathrm{reg}})$, respectively. The components of $\lambda^{\mathrm{Th}}$ Poisson-commute, so $\lambda^{\mathrm{Th}}\vert_{\Sigma^{\mathrm{Th}}}$ is invariant under the $\operatorname{U}(1)^c$-action on $\Sigma^{\mathrm{Th}}$. Since each torus element acts by a diffeomorphism, the chain rule shows that the rank of the differential of $\lambda^{\mathrm{Th}}\vert_{\Sigma^{\mathrm{Th}}}$ is constant along each torus orbit. The moment map $\mu:\mu^{-1}(\Sigma^{\mathrm{Th}})\longrightarrow\Sigma^{\mathrm{Th}}$ intertwines the two actions because the Hamiltonian vector fields of the components and their pullbacks are $\mu$-related. It follows that the first action preserves $\Sigma^{\mathrm{Th}}_{\mathrm{reg}}$ and the second preserves $\mu^{-1}(\Sigma^{\mathrm{Th}}_{\mathrm{reg}})$.
	\end{proof}
	
	If $c=b^{\Sigma}$, set $\lambda_{\mathrm{big}}^{\Sigma}\coloneqq\lambda^{\mathrm{Th}}$ and $\Sigma_{\mathrm{reg}}\coloneqq\Sigma^{\mathrm{Th}}_{\mathrm{reg}}$. The pair $(\lambda_{\mathrm{big}}^{\Sigma},\Sigma_{\mathrm{reg}})$ is a generalized Gelfand--Cetlin datum for $\Sigma$ precisely when $\Sigma^{\mathrm{Th}}_{\mathrm{reg}}$ is dense in $\Sigma$ and Condition (v) of Definition \ref{gen_gsdatum} holds.
	
	%%%%%%%%%%%%%%%%%%%%%%%%%%%%%%%%%%%%%%%%%%%%%%%%%%%%%%%%%%%%%%%%%%%
	%%%%%%%%%%%%%%%%%%%%%%%  SECTION 5  %%%%%%%%%%%%%%%%%%%%%%%%%%%%%%%
	%%%%%%%%%%%%%%%%%%%%%%%%%%%%%%%%%%%%%%%%%%%%%%%%%%%%%%%%%%%%%%%%%%%
	
	\section{The generalized abelianization theorem}
	In Section \ref{necessary_lemmas} we describe how the stabilizer of a moment value acts on its fiber. We use this description to prove the generalized abelianization theorem in Section \ref{abelianization_proof}.
	\subsection{Stabilizer abelianizations and moment map fibers}\label{necessary_lemmas}
	Suppose that $G$ is a compact connected Lie group with Lie algebra $\g$, and fix a generalized Gelfand--Cetlin datum $(\lambda_{\mathrm{big}}^{\Sigma},\Sigma_{\mathrm{reg}})$ for an orbit-type stratum $\Sigma\subset\g^*$. Let $(M,\omega,\mu)$ be a non-empty, connected Hamiltonian $G$-space with principal orbit-type stratum $\Sigma$. We use the induced Hamiltonian $\T_{\mathrm{big}}^\Sigma$-action on $M_{\mathrm{reg}}$ and moment map $\lambda_M$ from Section \ref{sec:gen-GC-data-definition}.
	
	\begin{lemma}\label{lem:collective-equivariance}
		The restriction $\mu\vert_{M_{\mathrm{reg}}}:M_{\mathrm{reg}}\longrightarrow\Sigma_{\mathrm{reg}}$ is $\T_{\mathrm{big}}^{\Sigma}$-equivariant. The subgroup $\T_{\mathrm{small}}^{\Sigma}$ acts trivially on $\Sigma_{\mathrm{reg}}$.
	\end{lemma}
	\begin{proof}
		The restriction $\mu\vert_{M_{\mathrm{reg}}}:M_{\mathrm{reg}} \longrightarrow\Sigma_{\mathrm{reg}}$ is Poisson. This follows by extending smooth functions on $\Sigma_{\mathrm{reg}}$ locally to $\g^*$ and using the induced Poisson bracket. It follows that the Hamiltonian vector fields of $\lambda_j\circ\mu$ on $M_{\mathrm{reg}}$ and of $\lambda_j$ on $\Sigma_{\mathrm{reg}}$ are $\mu$-related for each component $\lambda_j$ of $\lambda_{\mathrm{big}}^\Sigma$. If $m(t)$ is the Hamiltonian trajectory of $\lambda_j\circ\mu$ starting at $m\in M_{\mathrm{reg}}$, then $t\longmapsto\mu(m(t))$ is the Hamiltonian trajectory of $\lambda_j$ starting at $\mu(m)$. These flows generate the $\T_{\mathrm{big}}^{\Sigma}$-actions in Conditions (iv) and (vi) of Definition \ref{gen_gsdatum}. We conclude that $\mu$ is $\T_{\mathrm{big}}^{\Sigma}$-equivariant. The first $\ell^{\Sigma}$ components are $G$-invariant. They are therefore Casimirs for the Poisson structure on $\Sigma$, so their Hamiltonian vector fields vanish there. As these vector fields generate the $\T_{\mathrm{small}}^\Sigma$-action on $\Sigma_{\mathrm{reg}}$, this action is trivial.
	\end{proof}
	
	For $\xi \in \Sigma_{\mathrm{reg}}$, Condition (ii) of Definition \ref{gen_gsdatum} tells us that $\{(\mathrm{d} (\lambda_1\vert_\Sigma))_{\xi}, \dots, (\mathrm{d}(\lambda_{\ell^{\Sigma}}\vert_\Sigma))_{\xi}\}$ is a $\Z$-basis of the lattice $\Lambda_{\g_\xi/ [\g_\xi, \g_\xi]} \subset \g_\xi\big/ [\g_\xi, \g_\xi]$.
	As such, it forms a basis for $\g_\xi \big/ [\g_\xi, \g_\xi]$. 
	The coordinate map for this basis is
	\begin{align*}
		\kappa_\xi: \g_\xi \big/ [\g_\xi, \g_\xi] & \longrightarrow \R_{\mathrm{small}}^{\Sigma} \\  
		x_1(\mathrm{d} (\lambda_1\big\vert_\Sigma))_{\xi} + \dots + x_{\ell^{\Sigma}} (\mathrm{d} (\lambda_{\ell^{\Sigma}}\big\vert_\Sigma))_{\xi} &\longmapsto (x_1, \dots, x_{\ell^{\Sigma}}).
	\end{align*}
	
	\begin{proposition}\label{torus_iso}
		If $\xi \in \Sigma_{\mathrm{reg}}$, then $\kappa_\xi$ integrates to a Lie group isomorphism 
		$$
		\tau_\xi: G_\xi \big/ [G_\xi, G_\xi] \longrightarrow \T_{\mathrm{small}}^{\Sigma}.
		$$
	\end{proposition}
	\begin{proof}
		Let $v_j\coloneqq(\mathrm{d}(\lambda_j\vert_\Sigma))_\xi$ for $j=1,\dots,\ell^{\Sigma}$. By Condition (ii) of Definition \ref{gen_gsdatum} and our normalization of integral lattices,
		$$
		\ker\bigl(\exp:\g_\xi/[\g_\xi,\g_\xi]
		\longrightarrow G_\xi/[G_\xi,G_\xi]\bigr)
		=2\pi\Z v_1+\cdots+2\pi\Z v_{\ell^{\Sigma}}.
		$$
		The map $\kappa_\xi$ sends this lattice isomorphically onto $2\pi\Z^{\ell^{\Sigma}}$. This is the kernel of the exponential map $(x_1,\dots,x_{\ell^\Sigma})\longmapsto (e^{ix_1},\dots,e^{ix_{\ell^\Sigma}})$ from $\R^{\ell^\Sigma}$ to $\T_{\mathrm{small}}^\Sigma$. The map $\kappa_\xi$ therefore descends to the unique Lie group isomorphism $\tau_\xi$ whose differential is $\kappa_\xi$.
	\end{proof}
	Before stating the next proposition, we recall some pertinent facts. 
	The first is that $[G_\xi, G_\xi]$ acts trivially on $\mu^{-1}(\xi)$ for every $\xi$ in the principal orbit-type stratum $\Sigma\subset\g^*$ for $M$; see Proposition \ref{trivial_action}. 
	Another is that $\T_{\mathrm{big}}^{\Sigma}$ acts on $M_{\mathrm{reg}}$; see Definition \ref{gen_gsdatum}.
	\begin{proposition}\label{torus_iso_equiv}
		Suppose that $\xi \in \Sigma_{\mathrm{reg}}$.
		Then $[g] \cdot m = \tau_\xi([g]) \cdot m$ for all $g \in G_\xi$ and $m \in \mu^{-1}(\xi)$,
		where the left-hand side denotes the action of $G_\xi \big/ [G_\xi, G_\xi]$ induced by the action of $G$, and the right-hand side denotes the action of $\T_{\mathrm{small}}^{\Sigma} \subset \T_{\mathrm{big}}^{\Sigma}$.
	\end{proposition}
	\begin{proof}
		For $j=1,\dots,\ell^{\Sigma}$, choose the central representative $\widetilde v_j\in\z(\g_\xi)$ of $(\mathrm{d}(\lambda_j\vert_\Sigma))_\xi$. The chain rule and the Hamiltonian $G$-action imply
		$$
		X_{\lambda_j\circ\mu}(m)=(\widetilde v_j)_M(m)
		\quad\text{for all }m\in\mu^{-1}(\xi).
		$$
		By Proposition \ref{trivial_action}, changing the representative by an element of $[\g_\xi,\g_\xi]$ leaves the vector field on this level set unchanged. The infinitesimal actions of $G_\xi/[G_\xi,G_\xi]$ and $\T_{\mathrm{small}}^{\Sigma}$ consequently agree under $\kappa_\xi$. Since both tori are connected, their actions agree under $\tau_\xi$.
	\end{proof}
	
	\begin{corollary}\label{rho_lie_group_iso}
		The multiplication map
		\begin{align*}
			\rho_\xi: (G_\xi/[G_\xi,G_\xi])\times \T_{\mathrm{int}}^{\Sigma} &\longrightarrow \T_{\mathrm{big}}^{\Sigma}\\
			([g],t) &\longmapsto \tau_\xi([g])t
		\end{align*}
		is a Lie group isomorphism.
	\end{corollary}
	
	We now consider the action of $\T_{\mathrm{int}}^{\Sigma} \subset \T_{\mathrm{big}}^{\Sigma}$ on $M_{\mathrm{reg}}$.
	For a fixed $\xi\in\Sigma_{\mathrm{reg}}$, we give $\mu^{-1}(\xi)\times\T_{\mathrm{int}}^{\Sigma}$ the product action of $(G_\xi/[G_\xi,G_\xi])\times\T_{\mathrm{int}}^{\Sigma}$, with the second factor acting by multiplication on itself.
	
	\begin{proposition}\label{homeomorphism}
		For $\xi \in \Sigma_{\mathrm{reg}}$, the map 
		\begin{align*}
			\delta_\xi: \mu^{-1}(\xi) \times \T_{\mathrm{int}}^{\Sigma} &\longrightarrow \lambda_M^{-1}(\lambda_{\mathrm{big}}^{\Sigma}(\xi)) \\
			(m, t) &\longmapsto t \cdot m
		\end{align*}
		is a homeomorphism.
	\end{proposition}
	\begin{proof}
		Condition (v) of Definition \ref{gen_gsdatum} implies that
		$$
		a_\xi:\T_{\mathrm{int}}^{\Sigma}\longrightarrow
		\bigl(\lambda_{\mathrm{big}}^{\Sigma}\vert_{\Sigma_{\mathrm{reg}}}\bigr)^{-1}
		(\lambda_{\mathrm{big}}^{\Sigma}(\xi)),
		\qquad t\longmapsto t\cdot\xi
		$$
		is a diffeomorphism. For $m\in\lambda_M^{-1}(\lambda_{\mathrm{big}}^{\Sigma}(\xi))$, define $t(m)\coloneqq a_\xi^{-1}(\mu(m))$. Lemma \ref{lem:collective-equivariance} then gives
		$$
		\mu\bigl(t(m)^{-1}\cdot m\bigr)=\xi.
		$$
		It follows that $\delta_\xi$ has inverse
		$$
		m\longmapsto\bigl(t(m)^{-1}\cdot m,t(m)\bigr).
		$$
		The maps $\delta_\xi$ and $\delta_\xi^{-1}$ are continuous.
	\end{proof}
	
	\subsection{Proof of the abelianization theorem}\label{abelianization_proof}
	We prove the generalized abelianization theorem by verifying the conditions of Definition \ref{sss_isomorph}.
	Retain the notation and conventions of Section \ref{necessary_lemmas}. Let $(M,\omega,\mu)$ be a non-empty, connected Hamiltonian $G$-space with principal orbit-type stratum $\Sigma$, and fix $\xi \in \Sigma_{\mathrm{reg}}$.
	
	An isomorphism of stratified symplectic spaces between $M\sll{\xi}G$ and $M_{\mathrm{reg}}\sll{\lambda_{\mathrm{big}}^{\Sigma}(\xi)}\T_{\mathrm{big}}^{\Sigma}$ requires a bijection between the orbit-types of the $G_\xi$-action on $\mu^{-1}(\xi)$ and those of the $\T_{\mathrm{big}}^{\Sigma}$-action on $\lambda_M^{-1}(\lambda_{\mathrm{big}}^{\Sigma}(\xi))$; see Definition \ref{sss_isomorph}. We denote these sets by $\Stab_M(G,\xi)$ and $\Stab_{M_{\mathrm{reg}}}(\T_{\mathrm{big}}^{\Sigma},\lambda_{\mathrm{big}}^{\Sigma}(\xi))$, respectively. The first set can be described using the quotient torus $G_\xi/[G_\xi,G_\xi]$. For $p\in\mu^{-1}(\xi)$, the $G$-equivariance of $\mu$ implies $G_p\subset G_\xi$, so $(G_\xi)_p=G_p$. By Proposition \ref{trivial_action}, this stabilizer contains $[G_\xi,G_\xi]$ and therefore has image $G_p/[G_\xi,G_\xi]$ in this torus. The $\T_{\mathrm{big}}^\Sigma$-equivariance of $\mu\vert_{M_{\mathrm{reg}}}$ also constrains the stabilizers on the torus side. On $\lambda_M^{-1}(\lambda_{\mathrm{big}}^{\Sigma}(\xi))$, Lemma \ref{lem:collective-equivariance} gives $\mu(m)=(s,t)\cdot\mu(m)$ whenever $(s,t)\in\T_{\mathrm{small}}^{\Sigma}\times\T_{\mathrm{int}}^{\Sigma}$ fixes $m$. The torus $\T_{\mathrm{small}}^\Sigma$ acts trivially on $\Sigma_{\mathrm{reg}}$, so $\mu(m)=t\cdot\mu(m)$. The freeness of the $\T_{\mathrm{int}}^{\Sigma}$-action on $\Sigma_{\mathrm{reg}}$ forces $t=e$. It follows that every $\T_{\mathrm{big}}^{\Sigma}$-stabilizer on this level set has the form $H\times\{e\}$ for a closed subgroup $H\subset\T_{\mathrm{small}}^{\Sigma}$. We use $\tau_\xi$ to match the stabilizers in $G_\xi/[G_\xi,G_\xi]$ with the subgroups of $\T_{\mathrm{small}}^{\Sigma}$ that occur in this way. The following is a preliminary step in this direction.
	\begin{lemma}\label{normalization_lemma}
		If $p \in \mu^{-1}(\xi)$, then $G_\xi$ normalizes $G_p$.
	\end{lemma}
	\begin{proof}
		Suppose that $g \in G_\xi$.  
		Consider the subgroups $G_p/ [G_\xi, G_\xi]$ and $(gG_pg^{-1})/[G_\xi, G_\xi]$ of $G_\xi/[G_\xi, G_\xi]$.
		It follows that 
		$$
		[g](G_p \big/ [G_\xi, G_\xi]) [g]^{-1} = (gG_pg^{-1})/[G_\xi, G_\xi].
		$$
		As $G_\xi/[G_\xi, G_\xi]$ is abelian, we have
		$$
		G_p/[G_\xi, G_\xi] = (gG_pg^{-1})/[G_\xi, G_\xi].
		$$
		This immediately implies that $G_p = gG_pg^{-1}$.
	\end{proof}
	Each $G_\xi$-conjugacy class $(H)\in\Stab_M(G,\xi)$ is therefore a singleton. We identify it with $H$ and write $\mu^{-1}(\xi)_H=\mu^{-1}(\xi)_{(H)}$.
	On the other hand, recall the Lie group isomorphism $\tau_\xi:G_\xi/[G_\xi, G_\xi] \longrightarrow \T_{\mathrm{small}}^{\Sigma}$ from Proposition \ref{torus_iso}.
	\begin{proposition}\label{stabilizer_bijection}
		The map
		\begin{align*}
			\Stab_M(G, \xi) &\longrightarrow \Stab_{M_{\mathrm{reg}}}(\T_{\mathrm{big}}^{\Sigma}, \lambda_{\mathrm{big}}^{\Sigma}(\xi))\\
			H &\longmapsto \tau_\xi\left(H\big/ [G_\xi, G_\xi]\right) \times \{e\}
		\end{align*}
		is a bijection.
	\end{proposition}
	\begin{proof}
		For $p\in\mu^{-1}(\xi)$, write $H\coloneqq G_p\subset G_\xi$. Its stabilizer for the quotient torus $G_\xi/[G_\xi,G_\xi]$ is $H/[G_\xi,G_\xi]$. Proposition \ref{torus_iso_equiv} identifies the corresponding $\T_{\mathrm{small}}^{\Sigma}$-stabilizer with $\tau_\xi(H/[G_\xi,G_\xi])$, and the preceding argument shows that the $\T_{\mathrm{int}}^{\Sigma}$-component of its $\T_{\mathrm{big}}^{\Sigma}$-stabilizer is trivial. Every point of $\lambda_M^{-1}(\lambda_{\mathrm{big}}^{\Sigma}(\xi))$ is a $\T_{\mathrm{int}}^{\Sigma}$-translate of such a $p$ by Proposition \ref{homeomorphism}; its stabilizer is unchanged under this translation because $\T_{\mathrm{big}}^{\Sigma}$ is abelian. This proves surjectivity. Injectivity follows because subgroups of $G_\xi$ containing $[G_\xi,G_\xi]$ are determined by their images in the quotient, and $\tau_\xi$ is an isomorphism.
	\end{proof}
	
	For $p\in\mu^{-1}(\xi)$ with $G_p=H$, the stabilizer of $(p,t)$
	under the product action of $(G_\xi/[G_\xi,G_\xi])\times
	\T_{\mathrm{int}}^\Sigma$ is $(H/[G_\xi,G_\xi])\times\{e\}$.
	Proposition \ref{torus_iso_equiv} and the commutativity of
	$\T_{\mathrm{big}}^\Sigma$ give
	\[
	\delta_\xi(([g],s)\cdot(p,t))
	=\rho_\xi([g],s)\cdot\delta_\xi(p,t).
	\]
	\begingroup\emergencystretch=2em
	The homeomorphism $\delta_\xi$ therefore carries the locus with this
	stabilizer onto the locus with $\T_{\mathrm{big}}^\Sigma$-stabilizer
	$\tau_\xi(H/[G_\xi,G_\xi])\times\{e\}$.\par\endgroup
	
	\begin{proposition}\label{stratum_diffeo}
		If $H \in \Stab_M(G, \xi)$, then $\delta_\xi$ restricts to a diffeomorphism
		\[
		\mu^{-1}(\xi)_H \times \T_{\mathrm{int}}^{\Sigma} \longrightarrow \bigl(\lambda_M^{-1}(\lambda_{\mathrm{big}}^{\Sigma}(\xi))\bigr)_{\tau_\xi(H/[G_\xi, G_\xi]) \times \{e\}}.
		\]
	\end{proposition}
	\begin{proof}
		The preceding stabilizer calculation identifies the image of the source locus in the display. The formula for $\delta_\xi^{-1}$ in the proof of Proposition \ref{homeomorphism} is smooth on this image, and $\delta_\xi$ is smooth on $\mu^{-1}(\xi)_H\times\T_{\mathrm{int}}^{\Sigma}$. Their restrictions are therefore mutually inverse diffeomorphisms.
	\end{proof}
	
	We now prove the generalized abelianization theorem.
	
	\begin{theorem}[Generalized abelianization]\label{Theorem: Main result}
		Let $(M,\omega,\mu)$ be a non-empty, connected Hamiltonian $G$-space with principal orbit-type stratum $\Sigma\subset\g^*$.
		Suppose that $(\lambda_{\mathrm{big}}^{\Sigma},\Sigma_{\mathrm{reg}})$ is a generalized Gelfand--Cetlin datum for $\Sigma$.
		For $\xi \in \Sigma_{\mathrm{reg}}$, there is a canonical isomorphism of stratified symplectic spaces
		$$
		M \sll{\xi} G \simeq M_{\mathrm{reg}} \sll{\lambda_{\mathrm{big}}^{\Sigma}(\xi)}\T_{\mathrm{big}}^{\Sigma}.
		$$
	\end{theorem}
	\begin{proof}
		In light of Definition \ref{sss_isomorph}, it suffices to construct a homeomorphism $\varphi:M \sll{\xi} G \longrightarrow M_{\mathrm{reg}}\sll{\lambda_{\mathrm{big}}^{\Sigma}(\xi)} \T_{\mathrm{big}}^{\Sigma}$ and bijection $\phi:\Stab_M(G,\xi) \longrightarrow \Stab_{M_{\mathrm{reg}}}(\T_{\mathrm{big}}^{\Sigma},\lambda_{\mathrm{big}}^{\Sigma}(\xi))$ satisfying Conditions (i)--(iii) of that definition.
		The inclusion of $\mu^{-1}(\xi)$ into $\lambda_M^{-1}(\lambda_{\mathrm{big}}^{\Sigma}(\xi))$ induces the map
		\[
		\varphi([m]_{G_\xi})=[m]_{\T_{\mathrm{big}}^{\Sigma}}.
		\]
		We first show that this map is a homeomorphism.
		
		As shown before Proposition \ref{stratum_diffeo}, the homeomorphism
		$\delta_\xi$ intertwines the product action of
		$(G_\xi/[G_\xi,G_\xi])\times\T_{\mathrm{int}}^\Sigma$ with the
		$\T_{\mathrm{big}}^\Sigma$-action through $\rho_\xi$. Taking the
		quotients by $\T_{\mathrm{int}}^\Sigma$ gives the homeomorphism
		\[
		f:\mu^{-1}(\xi)\longrightarrow
		\lambda_M^{-1}(\lambda_{\mathrm{big}}^\Sigma(\xi))/
		\T_{\mathrm{int}}^\Sigma,\qquad p\longmapsto[p]_{\T_{\mathrm{int}}^\Sigma}.
		\]
		The map $f$ intertwines the $G_\xi/[G_\xi,G_\xi]$-action on its
		source with the residual $\T_{\mathrm{small}}^\Sigma$-action on its
		target through $\tau_\xi$. Taking these quotients gives the
		homeomorphism $\varphi$ because $G_\xi$ and $G_\xi/[G_\xi,G_\xi]$
		have the same orbits on $\mu^{-1}(\xi)$.
		This argument can be summarized in the figure below, where all diagrams commute:
		\begin{center}
			\begin{tikzcd}
				& (G_\xi/[G_\xi,G_\xi])\times \T_{\mathrm{int}}^{\Sigma} \arrow[d, phantom, "\circlearrowright \times \circlearrowright"] \arrow[rr, "\rho_\xi"]&& \T_{\mathrm{big}}^{\Sigma} \arrow[d, phantom, "\circlearrowright"] \\
				& \mu^{-1}(\xi) \times \T_{\mathrm{int}}^{\Sigma} \arrow[rr, "\delta_\xi"] \arrow[d] && \lambda_M^{-1}(\lambda_{\mathrm{big}}^{\Sigma}(\xi))\arrow[d] \\
				& \mu^{-1}(\xi) \arrow[rr, "f"]\arrow[d] &&\lambda_M^{-1}(\lambda_{\mathrm{big}}^{\Sigma}(\xi)) \big/ \T_{\mathrm{int}}^{\Sigma}\arrow[dd] \\
				& \mu^{-1}(\xi)\big/ \left(G_\xi \big/ [G_\xi, G_\xi] \right) \arrow[rrd, "\varphi"]\arrow[d, equal] &&\\
				M\sll{\xi}G\arrow[r,equal] &\mu^{-1}(\xi)\big/ G_\xi  \arrow[rr, "\varphi"] &&\lambda_M^{-1}(\lambda_{\mathrm{big}}^{\Sigma}(\xi)) \big/ \T_{\mathrm{big}}^{\Sigma} \arrow[r, equal] &  M_{\mathrm{reg}} \sll{\lambda_{\mathrm{big}}^{\Sigma}(\xi)} \T_{\mathrm{big}}^{\Sigma}  \,.
			\end{tikzcd}
		\end{center}
		
		Proposition \ref{stabilizer_bijection} gives the bijection $\phi$ referenced at the beginning of the proof. 
		At the same time, Proposition \ref{stratum_diffeo} implies that $\varphi$ restricts to a diffeomorphism 
		$$
		\varphi_{H}:\mu^{-1}(\xi)_{H}/G_\xi \longrightarrow \bigl(\lambda_M^{-1}(\lambda_{\mathrm{big}}^{\Sigma}(\xi))\bigr)_{\phi(H)}/\T_{\mathrm{big}}^{\Sigma}
		$$
		for all $H \in \Stab_M(G,\xi)$.
		
		Fix $H \in \Stab_M(G,\xi)$. 
		Write $(\omega_\xi)_H$ and $(\omega_{\lambda_{\mathrm{big}}^{\Sigma}(\xi)})_{\phi(H)}$ for the reduced symplectic forms on the source and target of $\varphi_H$, on each connected component, as in Theorem \ref{thrm:Sjamaar-Lerman}.
		It remains to show that $\varphi_H$ preserves these forms.
		
		By Proposition \ref{stratum_diffeo}, we have the following commutative diagram
		\begin{center}
			\begin{tikzcd}
				& M & \\
				\mu^{-1}(\xi)_{H}\arrow[d, "\pi"] \arrow[rr, "j", hookrightarrow]\arrow[ur, hookrightarrow, "k"] && \bigl(\lambda_M^{-1}(\lambda_{\mathrm{big}}^{\Sigma}(\xi))\bigr)_{\phi(H)} \arrow[d, "\theta"] \arrow[ul, hookrightarrow, "\iota", swap]\\
				\mu^{-1}(\xi)_{H}\big/G_\xi \arrow[rr, "\varphi_{H}"] && \bigl(\lambda_M^{-1}(\lambda_{\mathrm{big}}^{\Sigma}(\xi))\bigr)_{\phi(H)}\big/ \T_{\mathrm{big}}^{\Sigma} \, ,
			\end{tikzcd}
		\end{center}
		where $j$, $k$, and $\iota$ are the expected inclusions and $\pi$ and $\theta$ are the expected quotients.
		The defining identities for the reduced forms give
		\begin{align*}
			\pi^*\varphi_H^*\bigl((\omega_{\lambda_{\mathrm{big}}^{\Sigma}(\xi)})_{\phi(H)}\bigr)
			&= j^*\theta^*\bigl((\omega_{\lambda_{\mathrm{big}}^{\Sigma}(\xi)})_{\phi(H)}\bigr) \\
			&=j^*\iota^*\omega=k^*\omega=\pi^*\bigl((\omega_\xi)_H\bigr).
		\end{align*}
		Since $\pi$ is a surjective submersion, $\varphi_H$ is a symplectomorphism on each component.
	\end{proof}
	
	\section{Prescribed moment images and unitary Gelfand--Cetlin data}
	We realize each orbit-type stratum as the moment image of a
	Hamiltonian $G$-space and apply the abelianization theorem to a family
	of unitary examples.
	
	\subsection{Orbit-type strata as moment images}\label{sec:moment_map_image_submanifold}
	Fix a compact connected Lie group $G$ and an orbit-type stratum $\Sigma=G\cdot\sigma\subset\g^*$. The Hamiltonian $G$-space $(T^*G)_{\mathrm{impl},\sigma}$ of Section \ref{Subsection: Existence} has moment image $\Sigma$. The stabilizer $G_\xi$ acts transitively on $\mu_L^{-1}(\xi)$ for each $\xi\in\Sigma$, so all its reductions are points. We use the construction of the first author and Mayrand \cite{crooks2022Symplectic} to obtain another Hamiltonian $G$-space with moment image $\Sigma$ whose reduction at $\xi$ is diffeomorphic to $G_\xi\backslash G$.
	
	Consider a $G$-invariant submanifold $S \subset \mathfrak{g}^*$.
	For any $\xi \in S$, 
	\begin{equation}\label{eq:annoying_subalgebra}
		\mathfrak{h}_\xi \coloneqq \operatorname{ann}_{\mathfrak{g}}(T_\xi S) 
	\end{equation}
	is a Lie subalgebra of $\g$.
	Suppose that $\mathfrak{h}_\xi$ integrates to a connected closed subgroup $H_\xi \subset G$ for each $\xi \in S$.
	Use the left trivialization to identify $T^*G=G\times\mathfrak g^*$, and equip it with the symplectic form $-\mathrm{d}\theta$, as in Section \ref{Subsection: Existence}.
	We define the subset 
	$$
	\mathcal H_S \coloneqq \{(g,\xi) \in G \times S \mid g \in H_\xi \} \subset T^*G.
	$$
	The projection $\mathcal H_S \longrightarrow S$ onto the second factor has fiber $H_\xi \times \{\xi\}$ at $\xi$.
	On this fiber, multiplication is $(g,\xi)(h,\xi)=(gh,\xi)$, the identity is $(e,\xi)$, and inversion sends $(g,\xi)$ to $(g^{-1},\xi)$.
	If $\mathcal H_S$ is closed in $G \times S$, then it is an embedded submanifold of $G \times S$ and the projection $\mathcal H_S \longrightarrow S$ is a surjective submersion \cite[Theorems 4.17 and 4.18]{crooks2022Symplectic}.
	Fiberwise multiplication is smooth on $\mathcal H_S \times_S \mathcal H_S$. The identity section $\xi \longmapsto (e,\xi)$ and inversion on $\mathcal H_S$ are also smooth.
	The bundle $\mathcal H_S$ acts on $G \times S \longrightarrow S$ by $(g,\xi)\cdot(h,\eta)\coloneqq(hg^{-1},\xi)$ for $(g,\xi)\in\mathcal H_S$ and $(h,\eta)\in G\times S$ with $\xi=\eta$.
	As a set, the quotient $(G\times S)/\mathcal H_S$ is given by
	$$
	(G \times S) \big/ \mathcal H_S = \bigsqcup_{\xi \in S} G \big/H_\xi.
	$$
	Note also that left multiplication by $G$ on the first factor of $G \times S$ descends to an action of $G$ on $(G \times S)/ \mathcal H_S$.
	Crooks and Mayrand describe the symplectic structure and the Hamiltonian $G$-action on this quotient in \cite[Sections 4.3 and 4.5]{crooks2022Symplectic}.
	We state their result with the symplectic form and the moment map for the left $G$-action both negated to match our convention for $T^*G$.
	\begin{theorem}\label{reduction_submanifold_thrm}
		Let the submanifold $S\subset\g^*$ and the subset $\mathcal H_S\subset G\times S$ be as above. If $\mathcal H_S$ is closed in $G \times S$, then the following statements are true.
		\begin{enumerate}[(i)]
			\item The quotient $(G \times S)\big/\mathcal H_S$ has a unique smooth manifold structure for which the quotient map $\pi:G \times S \longrightarrow (G \times S)\big/\mathcal H_S$ is a surjective submersion. Write $j:G \times S\longrightarrow T^*G$ for the inclusion and $\omega$ for the symplectic form on $T^*G$. There is a unique symplectic form $\overline{\omega}$ on the quotient satisfying $\pi^*\overline{\omega}=j^*\omega$.
			\item The action of $G$ on $(G \times S)\big/\mathcal H_S$ is Hamiltonian with moment map 
			\begin{align*}
				\mu:(G\times S)\big/\mathcal H_S &\longrightarrow \mathfrak{g}^* \\
				[(g,\xi)] &\longmapsto \Ad_g^*(\xi).
			\end{align*}
		\end{enumerate}
	\end{theorem}
	
	We would like to apply this theorem to the case where $S$ is an orbit-type stratum $\Sigma \subset \g^*$.
	Applicability hinges on $\mathcal H_\Sigma$ being closed in $G \times \Sigma$, a condition that we now investigate.
	
	\begin{proposition}
		For all $\xi \in \Sigma$, $\mathfrak{h}_\xi = [\mathfrak{g}_\xi, \mathfrak{g}_\xi]$. 
	\end{proposition}
	
	\begin{proof}
		Taking annihilators in Lemma \ref{strata_cotangent_space} gives $\mathfrak h_\xi=[\g_\xi,\g_\xi]$.
	\end{proof}
	For $\xi \in \Sigma$, 
	$$
	H_\xi \coloneqq [G_\xi, G_\xi]
	$$
	is a connected closed subgroup of $G$ integrating $\h_\xi = [\g_\xi, \g_\xi]$.
	In keeping with previous notation, we define
	$$
	\mathcal H_\Sigma \coloneqq \{(g, \xi) \in G \times \Sigma \mid g \in [G_\xi, G_\xi]\}.
	$$
	
	\begin{proposition}
		The subset $\mathcal H_\Sigma$ is closed in $G \times \Sigma$. 
	\end{proposition}
	\begin{proof}
		Write $\Sigma=G\cdot\sigma$, and use the diffeomorphism
		$$
		(G/G_\sigma)\times\sigma\longrightarrow\Sigma,
		\qquad(gG_\sigma,\eta)\longmapsto\Ad_g^*(\eta).
		$$
		Local smooth sections of $G\longrightarrow G/G_\sigma$ give an open cover of $\Sigma$ by sets $V$ on which there is a smooth map $a:V\longrightarrow G$ satisfying
		$$
		[G_\eta,G_\eta]=a(\eta)[G_\sigma,G_\sigma]a(\eta)^{-1}
		\quad\text{for all }\eta\in V.
		$$
		The map
		$$
		G\times V\longrightarrow G\times V,\qquad
		(h,\eta)\longmapsto\bigl(a(\eta)ha(\eta)^{-1},\eta\bigr)
		$$
		is a diffeomorphism taking the closed embedded submanifold $[G_\sigma,G_\sigma]\times V$ onto $\mathcal H_\Sigma\cap(G\times V)$. Thus $\mathcal H_\Sigma\cap(G\times V)$ is closed in $G\times V$ for every member $V$ of the cover, and $\mathcal H_\Sigma$ is consequently closed in $G\times\Sigma$. The same argument also exhibits its smooth bundle structure.
	\end{proof}
	
	By Theorem \ref{reduction_submanifold_thrm}, $(G \times \Sigma)\big/\mathcal H_\Sigma$ is a Hamiltonian $G$-space, and the image of the moment map $\mu:(G\times \Sigma)/\mathcal H_\Sigma \longrightarrow \g^*$ is
	\begin{equation}\label{eq:moment_map_image_is_stratum}
		\mu\left((G\times \Sigma)\big/ \mathcal H_\Sigma\right) = \Sigma.
	\end{equation}
	For $\xi\in\Sigma$, the representatives of points of the moment fiber at $\xi$ have the form $(g,\Ad_{g^{-1}}^*\xi)$. Two such pairs represent the same point precisely when their first entries differ by left multiplication by an element of $[G_\xi,G_\xi]$. This identifies the moment fiber with $[G_\xi,G_\xi]\backslash G$ and its reduction with $G_\xi\backslash G$.
	
	The space $(G\times\Sigma)/\mathcal H_\Sigma$ is connected, and $\Sigma$ is its principal orbit-type stratum.
	
	\subsection{The case of \texorpdfstring{$\operatorname{U}(n)$}{U(n)}}\label{Subsection: Special case}
	Let $n$ be a positive integer, take $G=\operatorname{U}(n)$, and identify $\g^*$ with the space $\mathcal H(n)$ of Hermitian matrices as in Section \ref{sec:Unitary_and_tori}. Fix an orbit-type stratum $\Sigma\subset\mathcal H(n)$, and let $\sigma$ be the corresponding open face. By Equation \ref{eq:moment_map_image_is_stratum}, the Hamiltonian $G$-space $(G\times\Sigma)/\mathcal H_\Sigma$ has principal orbit-type stratum $\Sigma$. We first count the Gelfand--Cetlin functions on $\Sigma$ after accounting for the multiplicities of the eigenvalues. We then give an explicit generalized Gelfand--Cetlin datum for the ordered multiplicities $(n-1,1)$ when $n\geq3$.
	
	The stratum $\Sigma$ is specified by the ordered multiplicities $n_1,\dots,n_k$ of the distinct eigenvalues, listed in decreasing order. The stabilizer of a point in $\Sigma$ is conjugate to $\operatorname{U}(n_1)\times\cdots\times\operatorname{U}(n_k)$, and $\dim\sigma=k$. Since the commutator subgroup of $\operatorname{U}(n_i)$ is $\operatorname{SU}(n_i)$, the number of functions required by Definition \ref{gen_gsdatum} is
	$$
	b^\Sigma=\frac{1}{2}\left(n^2-\sum_{i=1}^k(n_i^2-1)+k\right)
	=\frac{1}{2}\left(n^2-\sum_{i=1}^k n_i^2\right)+k.
	$$
	
	For $\xi\in\mathcal H(n)$, write $\lambda_{j,r}(\xi)$ for the $r$th largest eigenvalue of the lower-right $(n-j)\times(n-j)$ submatrix of $\xi$, where $0\leq j\leq n-1$ and $1\leq r\leq n-j$. Reversing the standard basis conjugates this chain of unitary subgroups to the upper-left chain of Section \ref{sec:intro-abelianization}. When the matrix $\xi$ is understood, we write $\lambda_{j,r}$ for $\lambda_{j,r}(\xi)$. Consecutive rows satisfy the interlacing inequalities
	$$
	\lambda_{j,r}\geq\lambda_{j+1,r}\geq\lambda_{j,r+1}
	\qquad(0\leq j<n-1,\quad 1\leq r<n-j).
	$$
	On the locus of strict interlacing these functions give the classical Gelfand--Cetlin datum \cite[Section 3.5]{crooks2024Gelfand}. On $\Sigma$, the $n_i$ consecutive copies of the $i$th eigenvalue in the top row force $n_i-1$ copies in the next row, then $n_i-2$ in the following row, and so on. Keeping one top-row entry from each such triangle and all entries outside the triangles leaves
	$$
	\frac{n(n+1)}{2}-\sum_{i=1}^k\frac{n_i(n_i+1)}{2}+k
	=\frac{1}{2}\left(n^2-\sum_{i=1}^k n_i^2\right)+k
	=b^\Sigma
	$$
	entries, since $\sum_{i=1}^k n_i=n$. The following example verifies the conditions of Definition \ref{gen_gsdatum} for the resulting functions on the strata with ordered multiplicities $(n-1,1)$.
	
	The Gelfand--Cetlin system on a fixed unitary coadjoint orbit goes back to Guillemin--Sternberg \cite{guillemin1983gc}. Its use on non-regular orbits is developed, for example, in \cite{pabiniak2014width}. Here the two distinct eigenvalues vary over the open face, and Condition \textup{(vi)} requires a periodic lift to every Hamiltonian $G$-space with principal orbit-type stratum $\Sigma$.
	
	\begin{example}\label{ex:singular-u3-datum}
		Suppose that $n\geq3$ and let $\Sigma\subset\mathcal H(n)$ be the orbit-type stratum on which the larger eigenvalue has multiplicity $n-1$ and the smaller eigenvalue is simple. We have $\ell^\Sigma=2$, $u^\Sigma=n-1$, and $b^\Sigma=n+1$. By interlacing, the first $n-j-1$ entries of row $j$ all equal $\lambda_{0,1}$ for $0\leq j\leq n-1$. For instance, the array for $n=3$ is
		$$
		\begin{array}{ccccc}
			\lambda_{0,1} && \lambda_{0,2}=\lambda_{0,1} && \lambda_{0,3} \\
			& \lambda_{1,1}=\lambda_{0,1} && \lambda_{1,2} & \\
			&& \lambda_{2,1} &&
		\end{array}
		$$
		and the retained entries in general define the continuous map
		$$
		\lambda_{\mathrm{big}}^\Sigma
		\coloneqq(\lambda_{0,1},\lambda_{0,n},\lambda_{1,n-1},\ldots,\lambda_{n-1,1}):
		\mathcal H(n)\longrightarrow\R_{\mathrm{big}}^\Sigma.
		$$
		Let $\Sigma_{\mathrm{reg}}$ be the subset of $\Sigma$ defined by the strict inequalities
		$$
		\lambda_{0,n}<\lambda_{1,n-1}<\cdots<\lambda_{n-1,1}<\lambda_{0,1}.
		$$
		We claim that $(\lambda_{\mathrm{big}}^\Sigma,\Sigma_{\mathrm{reg}})$ is a generalized Gelfand--Cetlin datum for $\Sigma$.
		
		Fix $\xi\in\Sigma$ and choose a unit eigenvector $v=(v_1,\ldots,v_n)$ for $\lambda_{0,n}(\xi)$. The spectral decomposition is
		$$
		\xi=\lambda_{0,1}I_n-(\lambda_{0,1}-\lambda_{0,n})vv^*.
		$$
		The simple eigenvalue $\lambda_{0,n}$ and its eigenline $[v]$ vary smoothly on $\Sigma$, as does $\lambda_{0,1}=(\operatorname{tr}\xi-\lambda_{0,n})/(n-1)$. The spectral decomposition thus identifies $\Sigma$ with $\{(a,b)\in\R^2\mid a>b\}\times\mathbb{CP}^{n-1}$, with the coordinates $(\lambda_{0,1},\lambda_{0,n},[v])$. Under the trace pairing of Section \ref{sec:Unitary_and_tori}, the differentials of $\lambda_{0,1}\vert_\Sigma$ and $\lambda_{0,n}\vert_\Sigma$ are represented at $\xi$ by
		$$
		\frac{i}{n-1}(I_n-vv^*)\quad\text{and}\quad ivv^*,
		$$
		respectively. The determinants on the two eigenspaces identify $G_\xi/[G_\xi,G_\xi]$ with $\operatorname{U}(1)^2$. Their differentials send the two classes to $(i,0)$ and $(0,i)$, which form the standard integral basis of $\mathfrak u(1)^2$. Together with the $G$-invariance of these eigenvalue functions, this proves Conditions (i) and (ii).
		
		The trace of the lower-right $(n-j)\times(n-j)$ submatrix gives
		\begin{align*}
			\lambda_{j,n-j}
			&=\sum_{r=j+1}^n\xi_{rr}-(n-j-1)\lambda_{0,1}\\
			&=\lambda_{0,n}+(\lambda_{0,1}-\lambda_{0,n})\sum_{r=1}^{j}|v_r|^2
			\qquad (1\leq j\leq n-1),
		\end{align*}
		where the second equality follows from the spectral decomposition. The successive differences in the defining inequalities for $\Sigma_{\mathrm{reg}}$ thus give $(\lambda_{0,1}-\lambda_{0,n})|v_r|^2$ for $r=1,\ldots,n$. In particular, $\Sigma_{\mathrm{reg}}$ is the open dense locus on which every $v_r$ is nonzero. Its defining inequalities make it a union of fibers of $\lambda_{\mathrm{big}}^\Sigma\vert_\Sigma$. Every tuple satisfying the strict inequalities occurs. A smooth section over this open set is obtained by choosing the $v_r$ to be the positive square roots of the normalized differences.
		
		The $j$th circle of $\T_{\mathrm{int}}^\Sigma=\operatorname{U}(1)^{n-1}$ acts by conjugation with the diagonal matrix having $1$ in the first $j$ positions and $z_j$ in the remaining positions. The trace pairing gives $\sum_{r=j+1}^n\xi_{rr}$ as its moment map. The term $(n-j-1)\lambda_{0,1}$ is a Casimir on $\Sigma$, so $\lambda_{j,n-j}\vert_\Sigma$ is also a moment map for this circle. The product action multiplies the phase of $v_r/v_1$ by $z_1\cdots z_{r-1}$ for $r=2,\ldots,n$. These products determine $z_1,\ldots,z_{n-1}$ uniquely, so the torus acts freely and transitively on each fiber in $\Sigma_{\mathrm{reg}}$. The section and the phase coordinates give a smooth trivialization of the principal $\T_{\mathrm{int}}^\Sigma$-bundle and show that its projection is a submersion. The first two components of $\lambda_{\mathrm{big}}^\Sigma$ are Casimirs, which serve as moment components for the trivial $\T_{\mathrm{small}}^\Sigma$-action. This proves Conditions (iii)--(v).
		
		Now let $(M,\omega,\mu)$ be a non-empty, connected Hamiltonian $G$-space with principal orbit-type stratum $\Sigma$. For $m\in M_{\mathrm{reg}}$, take $\xi=\mu(m)$ and choose $v$ as above. The collective flows of $\lambda_{0,1}$ and $\lambda_{0,n}$ preserve $\mu$. The differential calculation and our generating vector field convention give the respective trajectories
		$$
		t\longmapsto\exp\!\left(-\frac{it}{n-1}(I_n-vv^*)\right)\cdot m,
		\qquad
		t\longmapsto\exp(-itvv^*)\cdot m.
		$$
		The second group element is the identity at time $2\pi$. At that time the first acts on $v^\perp$ by a scalar of determinant one and fixes $v$. It therefore lies in $[G_\xi,G_\xi]$ and fixes $m$ by Proposition \ref{trivial_action}. Both flows are thus $2\pi$-periodic. The diagonal circles on $M$ preserve $M_{\mathrm{reg}}$ and commute with these flows by $G$-invariance of the two eigenvalue functions. The two eigenvalue flows also commute because they preserve $\mu$. For each $j=1,\ldots,n-1$, the Hamiltonian $\lambda_{j,n-j}\circ\mu$ is the $j$th diagonal circle moment component minus $(n-j-1)\lambda_{0,1}\circ\mu$. Since $n-j-1$ is an integer, the resulting flows are complete, commute, and have time-$2\pi$ map equal to the identity. Together they give the $\T_{\mathrm{big}}^\Sigma$-action in Condition (vi).
		
		For $\xi\in\Sigma_{\mathrm{reg}}$, the reduction of $(G\times\Sigma)/\mathcal H_\Sigma$ at $\xi$ is diffeomorphic to $G_\xi\backslash G$ by Section \ref{sec:moment_map_image_submanifold}. For the ordered multiplicities $(n-1,1)$, the map $G_\xi g\longmapsto g^{-1}[v]$ identifies this quotient with $\mathbb{CP}^{n-1}$. Theorem \ref{Theorem: Main result} therefore identifies the non-abelian reduction with a $\operatorname{U}(1)^{n+1}$-reduction whose underlying manifold is $\mathbb{CP}^{n-1}$.
	\end{example}


\begin{thebibliography}{99}
		
		\bibitem{bredon1972introduction}
		\textsc{Bredon, G.~E.}
		\newblock \emph{Introduction to compact transformation groups}.
		\newblock Pure and Applied Mathematics, vol.~46. Academic Press, New York--London, 1972.
		
		\bibitem{crooks2022Symplectic}
		\textsc{Crooks, P., and Mayrand, M.}
		\newblock Symplectic reduction along a submanifold.
		\newblock \emph{Compos. Math.} \textbf{158} (2022), no.~9, 1878--1934.
		
		\bibitem{crooks2024Gelfand}
		\textsc{Crooks, P., and Weitsman, J.}
		\newblock Gelfand--Cetlin abelianizations of symplectic quotients.
		\newblock \emph{Pacific J. Math.} \textbf{333} (2024), no.~2, 253--271.
		
		\bibitem{duistermaat}
		\textsc{Duistermaat, J.~J., and Kolk, J.~A.~C.}
		\newblock \emph{Lie groups}.
		\newblock Universitext. Springer-Verlag, Berlin, 2000.
		
		\bibitem{gui-jef-sja:02}
		\textsc{Guillemin, V., Jeffrey, L., and Sjamaar, R.}
		\newblock Symplectic implosion.
		\newblock \emph{Transform. Groups} \textbf{7} (2002), no.~2, 155--184.
		
		\bibitem{guillemin1983gc}
		\textsc{Guillemin, V., and Sternberg, S.}
		\newblock The Gelfand--Cetlin system and quantization of the complex
		flag manifolds.
		\newblock \emph{J. Funct. Anal.} \textbf{52} (1983), no.~1, 106--128.
		
		\bibitem{guillemin1983collective}
		\textsc{Guillemin, V., and Sternberg, S.}
		\newblock On collective complete integrability according to the method of Thimm.
		\newblock \emph{Ergodic Theory Dynam. Systems} \textbf{3} (1983), no.~2, 219--230.
		
		\bibitem{hoffman2025stratified}
		\textsc{Hoffman, B., and Lane, J.}
		\newblock Stratified gradient Hamiltonian vector fields and collective integrable systems.
		\newblock \emph{J. Eur. Math. Soc.} (2025), published online first, doi:~\href{https://doi.org/10.4171/JEMS/1659}{10.4171/JEMS/1659}.
		
		\bibitem{lane2018Convexity}
		\textsc{Lane, J.}
		\newblock Convexity and Thimm's trick.
		\newblock \emph{Transform. Groups} \textbf{23} (2018), no.~4, 963--987.
		
		\bibitem{ler:95}
		\textsc{Lerman, E.}
		\newblock Symplectic cuts.
		\newblock \emph{Math. Res. Lett.} \textbf{2} (1995), no.~3, 247--258.
		
		\bibitem{lerman1998Nonabelian}
		\textsc{Lerman, E., Meinrenken, E., Tolman, S., and Woodward, C.}
		\newblock Non-abelian convexity by symplectic cuts.
		\newblock \emph{Topology} \textbf{37} (1998), no.~2, 245--259.
		
		\bibitem{mar-rat:86}
		\textsc{Marsden, J.~E., and Ratiu, T.}
		\newblock Reduction of Poisson manifolds.
		\newblock \emph{Lett. Math. Phys.} \textbf{11} (1986), no.~2, 161--169.
		
		\bibitem{marsden1974Reduction}
		\textsc{Marsden, J.~E., and Weinstein, A.}
		\newblock Reduction of symplectic manifolds with symmetry.
		\newblock \emph{Rep. Math. Phys.} \textbf{5} (1974), no.~1, 121--130.
		
		\bibitem{mik-wei:88}
		\textsc{Mikami, K., and Weinstein, A.}
		\newblock Moments and reduction for symplectic groupoids.
		\newblock \emph{Publ. Res. Inst. Math. Sci.} \textbf{24} (1988), no.~1, 121--140.
		
		\bibitem{molev2006Gelfand}
		\textsc{Molev, A.~I.}
		\newblock Gelfand--Tsetlin bases for classical Lie algebras.
		\newblock In \emph{Handbook of algebra}, vol.~4, Elsevier/North-Holland, Amsterdam, 2006, pp.~109--170.
		
		\bibitem{pabiniak2014width}
		\textsc{Pabiniak, M.}
		\newblock Gromov width of non-regular coadjoint orbits of
		$\mathrm{U}(n)$, $\mathrm{SO}(2n)$ and $\mathrm{SO}(2n+1)$.
		\newblock \emph{Math. Res. Lett.} \textbf{21} (2014), no.~1, 187--205.
		
		\bibitem{sjamaar1991Stratified}
		\textsc{Sjamaar, R., and Lerman, E.}
		\newblock Stratified symplectic spaces and reduction.
		\newblock \emph{Ann. of Math. (2)} \textbf{134} (1991), no.~2, 375--422.
		
	\end{thebibliography}
\end{document}